\documentclass[a4paper,notitlepage,twoside,leqno,12pt]{amsart}
\usepackage[all]{xy}
\usepackage{framed}
\usepackage{CJK}
\usepackage{mathrsfs}
\usepackage{anysize}  
\usepackage{enumerate}
\usepackage[curve]{xypic}
\usepackage{pifont}

\marginsize{3cm}{3cm}{3cm}{3cm}

\usepackage{dcolumn,indentfirst,color}
\usepackage[pdfpagemode=UseNone,pdfstartview=FitH]{hyperref}
\usepackage{amsmath,amssymb,amscd,amsthm,amsfonts,mathrsfs}
\usepackage{color,graphicx,xcolor,graphics}
\usepackage{titlesec} 
\usepackage{titletoc} 
\usepackage{mathrsfs}
\titlecontents{section}[20pt]{\addvspace{2pt}\filright}
{\contentspush{\thecontentslabel. }}
{}{\titlerule{\#}[8pt]{}\contentspage}
\usepackage{fancyhdr} 
\newtheorem{thm}{Theorem}[section]
\newtheorem{lem}[thm]{Lemma}
\newtheorem{cor}[thm]{Corollary}

\newtheorem{defi}{Definition}

\newcommand{\cbar}{\widehat{\mathbb{C}}}
\newcommand{\ra}{\rightarrow}
\newcommand{\pa}{\partial}
\newcommand{\sm}{\setminus}
\newcommand{\wh}{\widehat}
\newcommand{\ol}{\overline}
\newcommand{\mc}{\mathcal}

\newcommand{\wt}{\widetilde}
\newcommand{\G}{\Gamma}
\newcommand{\g}{\gamma}

\newcommand{\FFF}{\mathcal{F}}
\newcommand{\JJJ}{\mathcal{J}}
\newcommand{\PPP}{\mathcal{P}}
\newcommand{\QQQ}{\mathcal{Q}}

\newcommand{\AAA}{\mathcal{A}}

\newcommand{\CCC}{\mathcal{C}}
\newcommand{\UUU}{\mathcal{U}}
\newcommand{\NNN}{\mathcal{N}}

\newcommand{\CC}{\mathscr{C}}
\newcommand{\UU}{\mathcal{U}}
\newcommand{\DD}{\mathscr{D}}

\newcommand{\OO}{\mathscr{O}}

\theoremstyle{definition}

\makeatletter\@addtoreset{equation}{section}\makeatother 

\titleformat{\section}{\large}{\textbf{\thesection.}}{1em}{\textbf}
\titleformat{\subsection}{\normalsize}{\textbf{\thesubsection.}}{1em}{\textbf}
\titleformat{\subsubsection}{\normalsize}{\thesubsubsection.}{1em}{\textbf}
\titlespacing*{\subsubsection}{0pt}{*0.5}{*0.5}

\begin{document}
\date{}
\title{Decomposition of rational maps \\by stable multicurves}
\author{Guizhen Cui}
\address{Guizhen Cui, School of Mathematical Sciences, Shenzhen University, Shenzhen, 518052, P. R. China, \& Academy of Mathematics and Systems Science Chinese Academy of Sciences, Beijing, 100190, P. R. China}
\email{gzcui@szu.edu.cn}
\author{Fei Yang}
\address{Fei Yang, School of Mathematics, Nanjing University, Nanjing, 210093, P. R. China}
\email{yangfei@nju.edu.cn}
\author{Luxian Yang}
\address{Luxian Yang, School of Mathematical Sciences, Shenzhen University, Shenzhen, 518052, P. R. China}
\email{lxyang@szu.edu.cn}
	
\thanks{The first author is supported by National Key R\&D Program of China No.2021YFA1003203, and NSFC grants No.12131016 and No.12071303. The second author is supported by NSFC grant No.12222107. The third author is supported by Basic and Applied Basic Research Foundation of Guangdong Province Grants No.2023A1515010058.}
\subjclass[2000]{}

\begin{abstract}
A completely stable multicurve of a post-critically finite rational map induces a combinatorial decomposition. The projections of the small Julia sets are immersed within the original Julia set. We prove that two small Julia sets are disjoint if and only if they are separated by a coiling curve. Furthermore, we prove that a post-critically finite rational map with a coiling curve is renormalizable. 

Using a similar argument, we give a sufficient condition for a Fatou domain to qualify as a Jordan domain. By tuning polynomails in such a Fatou domain, we provide examples of post-critically finite rational maps with coiling curves. 
\end{abstract}
\maketitle

\section{Introduction}
Stable multicurves were introduced by Thurston in his work on characterizing post-critically finite rational maps \cite{douady1993proof}. For a post-critically finite branched covering, a stable multicurve is a Thurston obstruction if it leads to a matrix whose maximal eigenvalue is at least one (see \S \ref{sec:thurston} for details). Thurston proved that a post-critically finite branched covering map with hyperbolic orbifold is Thurston equivalent to a rational map if and only if it has no Thurston obstructions. 


Pilgrim throughly studied the decomposition and combination of post-critically finite branched coverings of the Riemann sphere using stable multicurves \cite{pilgrim2003combinations}, leading to a canonical decomposition theorem. Furthermore, he conjectured that this canonical decomposition is determined by the canonical Thurston obstruction, whose existence is equivalent to the existence of Thurston obstructions. Selinger subsequently provided a positive answer to this conjecture \cite{selinger2012thurston}. 

In this paper, we focus on the decomposition of post-critically finite rational maps. These maps have no Thurston obstructions, implying that the canonical Thurston obstruction is an empty set of curves. When a map $f$ possesses a stable multicurve, we can cut the sphere along the multicurve and define a dynamics on the collection of the pieces naturally. The first-return map on a periodic piece can be extended to a new post-critically finite branched covering that also has no Thurston obstructions. So a periodic piece induces a rational map. A natural question is to explore the relationships between the dynamics on these pieces. Before stating our results, we outline some basic concepts. 

Let $F$ be a branched covering of the Riemann sphere $\cbar$, with the assumption that its degree $\deg F$ is greater than $1$. A point $z\in\cbar$ is a {\bf critical point} of $F$ if $F$ is not injective in any neighborhood of $z$. We denote by $\CCC_F$ the set of critical points of $F$. The {\bf post-critical set} of $F$ is 
$$
\PPP_F=\ol{\bigcup_{n \geq 1}F^n(\CCC_F)}.
$$
The map $F$ is {\bf post-critically finite (PCF)} if $\PPP_F$ is a finite set. By a {\bf marked} branched covering $(F,\PPP)$ we mean a PCF branched covering $F$ of $\cbar$ with a finite marked set $\PPP\subset\cbar$ such that $\PPP_F\cup F(\PPP)\subset\PPP$. In the case where $\PPP=\PPP_F$, we simplify the notation by writing $(F,\PPP)=F$. 

A Jordan curve $\g\subset\cbar\sm\PPP$ is {\bf essential} if each of the two components of $\cbar\sm\g$ contains at least two points in $\PPP$. A {\bf multicurve} $\G$ is a non-empty, finite collection of essential Jordan curves that are pairwise disjoint and pairwise non-isotopic rel $\PPP$. A multicurve $\G$ is {\bf stable} if each essential curve in $F^{-1}(\G)$ is isotopic rel $\PPP$ to a curve in $\G$; or is {\bf pre-stable} if each curve in $\G$ is isotopic rel $\PPP$ to a curve in $F^{-1}(\G)$; or is {\bf completely stable} if it is both stable and pre-stable.

Let $f$ be a PCF rational map of degree $\deg f\geq 2$. Beside being viewed as a branched covering, it induces dynamical systems on $\cbar$ by iteration. The Fatou set $\FFF_f$ and the Julia set $\JJJ_f$ are defined based on whether the iterating sequence $\{f^n\}_{n\geq 1}$ forms a normal family in a neighborhood of the point or not (see e.g. \cite{beardon1991iteration, milnor2006dynamics} for their definitions and basic properties). By definition, $f^{-1}(\FFF_f)=\FFF_f$ and $f^{-1}(\JJJ_f)=\JJJ_f$, leading to a decomposition of $\cbar$ into two completely invariant sets. 

Every completely stable multicurve provides a combinatorial decomposition of $f$ in the following manner. For any completely stable multicurve $\G$ of $f$, denote by $\UU_{\G}$ the collection of the components of $\,\cbar\sm\G$. The map $f$ induces a map $f_{\G}$ on $\UU_{\G}$, and each $U\in\UU_{\G}$ is eventually periodic under $f_{\G}$ (refer to \S \ref{sec:combinatorial}). For periodic components of $f_{\G}$, we obtain the following result.

\begin{thm}\label{thm:combinatorial}
Let $f$ be a PCF rational map and $\,\G$ be a completely stable multicurve of $f$. Suppose $U\in\UU_{\G}$ is periodic with period $p\geq 1$. Then 
\begin{itemize}
\item[(1)] there exists a PCF rational map $g$ and two orientation-preserving homeomorphisms $\varphi_0,\varphi_1$ of $\,\cbar$ such that 
\begin{itemize}
\item[(a)] $\varphi_1$ is isotopic to $\varphi_0$ rel $\PPP:=(\PPP_f\cap U)\cup\PPP_U$, where $\PPP_U$ is a finite subset of $\,\cbar\sm U$ such that $\PPP_U\cap D$ is a singleton for each component $D$ of $\cbar\sm U$; 
\item[(b)] $\varphi_0\circ f^p=g\circ\varphi_1$ on $U^p$, where $U^p$ is the unique component of $f^{-p}(U)$ parallel to $U$ rel $\PPP$. 
\end{itemize}
Moreover, the map $g$ is unique up to holomorphic conjugation. 
\item[(2)] there exists a continuous map $\pi:\JJJ_g\to\JJJ_f$ such that $f^p\circ\pi=\pi\circ g$ on $\JJJ_g$.
\end{itemize}
\end{thm}
For a connected set $U\subset\cbar\sm\PPP$, a complementary component of $U$ is {\bf simple type} if it intersects $\PPP$ at no more than one point. Let $\wh{U}$ be the union of $U$ with its simple type complementary components. Two sets $U$ and $V$ are said to be {\bf parallel} to each other rel $\PPP$ if $\,\wh{U}$ is isotopic to $\wh{V}$ rel $\PPP$. The rational map $g$ in Thoerem \ref{thm:combinatorial} is called the { \bf combinatorial renormalization} of $f$ on $U$. 

For each periodic $U_i\in\UU_{\G}$, by Theorem \ref{thm:combinatorial}, we obtain a rational map $g_i$ and a continuous map $\pi_i:\JJJ_{g_i}\to\JJJ_f$. Denote by $\JJJ_i=\pi_i(\JJJ_{g_i})$. If $U_i\in\UU_{\G}$ is eventually periodic but not periodic, let $k_i\geq 1$ be the minimal integer such that $U_j:=f_{\G}^{k_i}(U_i)$ is periodic. We associate to $U_i$ the preimage of $f^{-k_i}(\JJJ_j)$ that is parallel to $U_i$ and denote it by $\JJJ_{i}$. Thus each $U_i\in\UU_{\G}$ induces a subset $\JJJ_{i}$ of $\JJJ_f$. A natural question is how the sets $\JJJ_{i}$ relate to each other and to the original Julia set. Specifically, when is $\JJJ_1\cap\JJJ_2$ an empty set and when is the map $\pi$ injective?
\vskip 0.24cm

 The following concept is introduced precisely for solving the above question. For any $\g\in\G$ and any $n\ge 1$, denote 
$$
\CC_n(\g,\G):=\{\text{the components of }f^{-n}(\G) \text{ which are isotopic to }\g \text{ rel }\PPP_f\}
$$ and
$$
\kappa_n(\g,\G)=\#\CC_n(\g,\G).
$$
\begin{defi}\label{def:coiling}
A curve $\g\in\G$ is a {\bf coiling curve of $\,\G$} or is {\bf coiling in $\G$} if $\kappa_n(\g,\G)\to\infty$ as $n\to\infty$. 
\end{defi}


A stable multicurve $\G$ of $f$ is a {\bf Cantor multicurve} if every curve $\g\in\G$ is coiling in $\G$. For its properties, we refer to \S 2 of \cite{cui2016renormalizations}. If $\G$ is a Cantor multicurve, then for each piece $U_i\in\UU_{\G}$, the map $\pi_i:\JJJ_{g_i}\to\JJJ_f$ is injective, and $\JJJ_1\cap\JJJ_2=\emptyset$ for any distinct pieces $U_1,U_2\in\UU_{\G}$ \cite{cui2016renormalizations}. The fact that every curve in $\G$ is a coiling curve of $\G$ is crucial for establishing these results. Building on the partially similar idea in \cite{cui2016renormalizations}, we prove the following theorem which implies that coiling curves partition the Julia set into disjoint pieces. 

\begin{thm}\label{thm:disjoint}
Let $f$ be a PCF rational map and $\,\G$ be a completely stable multicurve of $f$. Then there is a finite marked set $\PPP$ containing $\PPP_f$ with $f(\PPP)\subset\PPP$ and a completely stable multicurve $\G_{\PPP}$ of $(f,\PPP)$, whose isotopy classes include those of $\G$ rel $\PPP$,  such that for any $U_1\neq U_2\in\UU_{\G_{\PPP}}$ with a common boundary $\g\in\G_{\PPP}$,  $\JJJ_1\cap\JJJ_2=\emptyset$ if and only if $\g$ is a coiling curve of $\,\G_{\PPP}$. 
\end{thm}
Here we give some explaination of the marked set $\PPP$ and the multicurve $\G_{\PPP}$. A curve $\g\in\G$ is {\bf periodic} if there is an integer $p\geq 1$ such that a component $\g^p$ of $f^{-p}(\g)$ is isotopic to $\g$ rel $\PPP_f$. There is an important fact that if a periodic curve $\g\in\G$ is not coiling, then $\kappa_n(\g,\G)=1$ for all $n\geq 1$ (Lemma \ref{lem:unique}). For curves in $\G$ that are neither periodic nor coiling in $\G$, there is an integer $N$, such that $\kappa_n(\g,\G)\leq N$ for all $n\geq 1$. So by suitably adding a finite number of marked points (see \S \ref{sec:proof1.2}), we can get a marked set $\PPP\supset\PPP_f$ and a new multicurve $\G_{\PPP}$ of $(f,\PPP)$ such that
\begin{itemize}
\item[(1)] either $\kappa_n(\g,\G_{\PPP})=1$ for all $n\geq 1$ or $\g$ is a coiling curve in $\G_{\PPP}$. 
\item[(2)] the isotopy classes of $\G_{\PPP}$ rel $\PPP$ contains the isotopy classes of $\G$ rel $\PPP$,
\item[(3)] the isotopy classes of coiling curves of $\G_{\PPP}$ rel $\PPP_f$ equal the isotopy classes of coiling curves of $\G$ rel $\PPP_f$. 
\end{itemize}
So the partition of $\JJJ_f$ derived from $\G_{\PPP}$ implies the partition of $\JJJ_f$ obtained from $\G$. 

\vskip 0.24cm
Renormalization characterizes dynamical sub-systems through first-return maps on subsets. For polynomials, renormalizations can be defined by polynomial-like maps, which were introduced by Douady and Hubbard to study the self-similarity of the Mandelbrot set \cite{douady1985dynamics}. 
The renormalization of quadratic polynomials has been extensively studied in works such as \cite{hubbard1993local, kahn2006priori, kahn2008priori, kahn2009priori, lyubich1997dynamics, mcmullen1994complex, milnor2000local}. Similarly, for rational maps, renormalizations can be defined using rational-like maps in the same way (refer to \cite{buff2003virtually, cui2016renormalizations} or see \S \ref{sec:rational-like}). Typically, the renormalizations of polynomials are detected via Yoccoz puzzles, which are constructed from external rays and equipotential curves. However, effective methods for detecting the renormalizations of rational maps are still lacking in general. 

As stated in Theorem \ref{thm:combinatorial}, for any completely stable multicurve, the first-return maps on periodic components always yield combinatorial renormalizations of PCF rational maps. More work is required to remove the term ``combinatorial". For instance, in the case of a rational map which is the mating of polynomials, the equator forms a completely stable multicurve. However the first-return map on each complementary component of the equator does not induce a (rational-like) renormalization. On the other hand, if a rational map has Cantor multicurves, it is always renormalizable \cite[Theorem 1.3]{cui2016renormalizations}. Theorem \ref{thm:disjoint} implies that adjacent small Julia sets are disjoint if the corresponding pieces of $\cbar\sm\G$ are separated by coiling curves. This leads to a question: if a small Julia set is disjoint with all its adjacent small Julia sets, can it be the Julia set of a rational-like map? We prove that it is indeed true in our case. 

\begin{thm}\label{thm:renorm}
Suppose that $U\in\UU_{\G}$ is a periodic component with period $p\geq 1$. If every curve in $\pa U$ is coiling in $\,\G$, then there exists a finitely connected domain $V$, isotopic to $U$ rel $\PPP_f$, and a component $V^p$ of $f^{-p}(V)$ such that $f^p: V^p\to V$ is a renormalization. Moreover, $f^p:V^p\to V$ is hybrid equivalent to the combinatorial renormalization of $f$ on $U$. 
\end{thm}

If $\,\G$ contains a Cantor multicurve $\G'$, then for each periodic piece $U\in\UU_{\G'}$, each curve in $\pa U$ is coiling in $\,\G'$. In fact, in \cite{cui2016renormalizations}, the authors demonstrated that a Cantor multicurve induces an exact annular system, leading to two domains where one is compactly contained in the other. In general, the multicurve does not induce exact annular systems. The proof of Theorem \ref{thm:renorm} relies on the expanding property of PCF rational maps on their Julia sets. We find the renormalization domain by locating a continuum whose inverse images either coincide with itself or do not intersect it. 

For a completely stable multicurve that does not contain any Cantor multicurves, it is possible that no periodic component $U\in\UU_{\G}$ satisfying the conditions in Theorem \ref{thm:renorm}. Nevertheless, we find that as long as $\G$ contains coiling curves, there is a submulticurve $\G'\subset\G$, which is certainly not a Cantor multicurve, such that there is a periodic component in $\UU_{\G'}$ satisfying the condition in Theorem \ref{thm:renorm}. In other words, we get the following theorem. 

\begin{thm}{\label{thm:submulticurve}}
A PCF rational map with a coiling curve is renormalizable. Moreover each boundary component of the renormalization domains is isotopic to a coiling curve relative to the post-critical set. 
\end{thm}

Let $f$ be a PCF rational map. A Jordan curve $\beta\subset\cbar\sm\PPP_f$ is {\bf peripheral} around a point $a\in\PPP_f$ if one component of $\,\cbar\sm\beta$ intersects $\PPP_f$ at the point $a$. By incorporating peripheral curves into multicurves, we introduce the following concept. 
\begin{defi}\label{def:coilingdomain}
A fixed Fatou domain $D$ of a PCF rational map $f$ is a {\bf coiled Fatou domain} if there is an essential curve $\alpha\subset\cbar\sm\PPP_f$ such that
\begin{itemize}
\item[(1)] $f^{-1}(\alpha)$ contains a curve $\alpha^1$ isotopic to $\alpha$ rel $\PPP_f$.
\item[(2)] $f^{-1}(\alpha)$ contains curves isotopic rel $\PPP_f$ to $\beta$ which is peripheral around the fixed point $a\in D$.
\end{itemize}
Here we use the term ``coiled" since the curve $\beta$ is regarded as a coiling curve although it is non-essential. We say $D$ is {\bf coiled by a polynomial} if the combinatorial renormalization of $f$ on $U$ is a polynomial, where $U$ is the component of $\,\cbar\sm\alpha$ disjoint from $a$. 
\end{defi}
\begin{figure}[htbp]
\centering
\includegraphics[width=9cm]{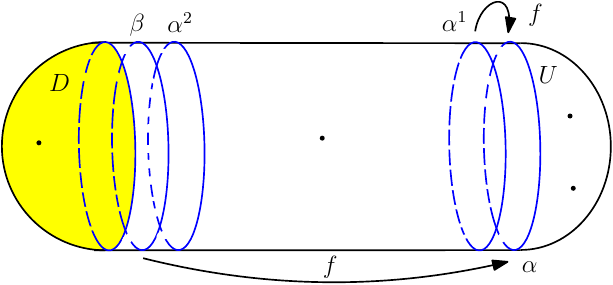}
\caption{The yellow disc $D$ is a Fatou domain. The curve $\alpha^2$ is a component of $f^{-1}(\alpha)$ isotopic to $\beta$ rel $\PPP_f$. }
\label{fig:Jgcoloring.pdf}
\end{figure}

\begin{thm}\label{thm:coiledJordan}
A coiled Fatou domain of PCF rational maps is always a Jordan domain. Moreover, its closure is disjoint from the closure of other Fatou domains. 
\end{thm}

To construct a rational map with a coiling curve (which is required by Theorem \ref{thm:submulticurve}), we first construct rational maps with a coiled Fatou domain coiled by a polynomial which is injective on its Hubbard tree (refer to \S \ref{sec:coiled} for definitions) by tuning and disc-annulus surgery.  We then prove that such a coiled Fatou domain is tunable with any PCF polynomials (Theorem \ref{thm:tunable}). It is straightforward to verify that the resulting rational maps possess a coiling curve. In \S \ref{sec:specific}, we provide two specific examples of PCF rational maps with a coiled Fatou domain coiled by the polynomial $z\mapsto z^2-1$, which is injective on its Hubbard tree. 



\section{Combinatorial decomposition}
\subsection{Thurston equivalence}\label{sec:thurston}
Before we prove Theorem \ref{thm:combinatorial}, we first review some known results. Thurston provided a complete topological characterization of PCF rational maps \cite{douady1993proof}. He defined Thurston equivalence between PCF branched coverings and established a necessary and sufficient condition for a PCF branched covering map to be Thurston equivalent to a rational map using multicurves. In this section, we will recall related concepts and results. 


A marked branched covering $(F,\PPP)$ is {\bf Thurston equivalent} to a marked rational map $(f,\QQQ)$ if there exist orientation preserving homeomorphisms $(\phi_0,\phi_1)$ of $\cbar$ such that $\phi_0$ is isotopic to $\phi_1$ rel $\PPP$ and $\phi_0\circ F=f\circ\phi_1$.

Let $\G$ be a multicurve of $(F,\PPP)$. Its transition matrix $M_{\G}=(a_{\g\beta})$ is defined as
$$
a_{\g\beta}=\sum_{\delta}\frac{1}{\deg(F:\,\delta\to\beta)},
$$
where the summation is taken over all components of $F^{-1}(\beta)$ that are isotopic to $\g$ rel $\PPP$. Denote by $\lambda(M_{\G})$ the leading eigenvalue of $M_{\G}$. The multicurve $\G$ is a {\bf Thurston obstruction} of $(F,\PPP)$ if $\lambda(M_{\G})\ge 1$.

\begin{thm}[{\bf Thurston}]\label{thm:thurston}
Let $(F,\PPP)$ be a marked branched covering of $\,\cbar$ with hyperbolic orbifold. Then $(F,\PPP)$ is Thurston equivalent to a marked rational map if and only if $(F,\PPP)$ has no Thurston obstructions. Moreover, the rational map is unique up to holomorphic conjugation. 
\end{thm}

Refer to \cite{buff2014teichmuller, douady1993proof} for the definition of hyperbolic orbifold and the proof of the theorem. The following lemma is useful for determing whether a stable multicurve is a Thurston obstruction (refer to \cite[Theorem B.6]{mcmullen1994complex}). 

\begin{lem}\label{lem:irreducible}
For any stable multicurve $\G$ with $\lambda(M_{\G})>0$, there is an irreducible multicurve $\G_0\subset\G$ such that $\lambda(M_{\G_0})=\lambda(M_{\G})$.
\end{lem}

A multicurve $\G$ of $(F,\PPP)$ is {\bf irreducible} if for each pair $(\g,\beta)\in\G\times\G$, there exists a sequence of curves
$$
\{\g=\alpha_0,\, \alpha_1,\cdots,\alpha_n=\beta\}
$$
in $\G$ such that for $1\le k\le n$, $F^{-1}(\alpha_k)$ has a component isotopic to $\alpha_{k-1}$ rel $\PPP$. Note that an irreducible multicurve is pre-stable. If $\g=\beta$, then $\g$ is periodic and the sequence of curves $\alpha_0=\g,\alpha_1,\cdots,\alpha_{n-1}$ is called {\bf a cycle containing $\g$}. It is necessary to remark that, in general, the cycle containing a periodic curve is not unique. 

A cycle $\{\g=\alpha_0,\, \alpha_1,\cdots,\alpha_{n-1}\}$ is a {\bf Levy cycle} if it is a multicurve and for $i=0,1,\cdots,n-1$, $F^{-1}(\alpha_{i+1})$ contains a component $\alpha'_{i}$ that is isotopic to $\alpha_i$ rel $\PPP$, and $F:\alpha'_{i}\to\alpha_{i+1}$ is a homeomorphism. By definition, if $F$ has a Levy cycle, then $F$ has a Thurston obstruction. 

\vskip 0.24cm

The following theorem demonstrates the existence of a semi-conjugacy from $F$ to $f$ that satisfies certain special properties provided that $F$ is holomorphic in a neighborhood of the critical cycles (refer to Appendix A in \cite{cui2016renormalizations}). 

\begin{thm}[{\bf Rees-Shishikura}]\label{thm:semiconjugacy}
Suppose that $F$ is a PCF branched covering of $\,\cbar$ which is Thurston equivalent to a rational map $f$ through the pair $(\psi_0, \psi_1)$. If $F$ is holomorphic in a neighborhood $D$ of the critical cycles of $F$, then there exists a sequence of homeomorphisms $\{\phi_n\}_{n\geq 0}$ of $\,\cbar$ isotopic to $\psi_0$ rel $\PPP_{F}$, such that
\begin{itemize}
\item[(1)] $\phi_0$ is holomorphic in $D$ and $\phi_n|_{D}=\phi_0|_{D}$;
\item[(2)] $\phi_n\circ F=f\circ\phi_{n+1}$ on $\cbar$;
\item[(3)] $\phi_n$ converges uniformly to a continuous onto map $h$ and $h\circ F=f\circ h$ on $\cbar$;
\item[(4)] for all $z\in\cbar$, $h^{-1}(z)$ is a full continuum, and moreover, $h^{-1}(z)$ is a singleton if $z\in\FFF_f$.
\item[(5)] for points $x,y\in\cbar$ with $f(x)=y$, $h^{-1}(x)$ is a component of $F^{-1}(h^{-1}(y))$. Furthermore, $\deg F|_{h^{-1}(x)}=\deg_xf$. 
\end{itemize}
\end{thm}

\subsection{Proof of Theorem \ref{thm:combinatorial}}\label{sec:combinatorial}
Let $(f,\PPP)$ be a marked rational map. A domain or a continuum $E\subset\cbar$ is {\bf simple type} if there is a simply connected domain $D\subset\cbar$ such that $E\subset D$ and $\#(D\cap\PPP)\le 1$; or {\bf annular type} if $E$ is not simple type and there is an annulus $A\subset\cbar\sm\PPP$ such that $E\subset A$; or {\bf complex type} otherwise. For any complex type domain $U$, we denote by $\wh{U}$ the union of $U$ and all the simple type components of $\cbar\sm U$. Two sets $U$ and $V$ are called {\bf parallel} to each other rel $\PPP_f$ if $\wh{U}$ is isotopic to $\wh{V}$ rel $\PPP_f$. 

Now we take $\PPP=\PPP_f$ and $\G$ a completely stable multicurve of $(f,\PPP_f)$. Then each component $U$ of $\cbar\sm\G$ is complex type since the sum of the number of connected components of $\pa U$ and the number of points in $\PPP_f\cap U$ is greater than $2$. Denote by $\UU:=\UU_{\G}$ the collection of components of $\cbar\sm\G$, and by $\UU^n$ the collection of all the complex type components of $\cbar\sm f^{-n}(\G)$. 

Define a map $f_{\G}$ on $\UU$ by $f_{\G}(U)=V$ if $f(U^1)=V$, where $U^1\in\UU^1$ is the element parallel to $U$ rel $\PPP_f$. Since $\G$ is completely stable, for each component $U\in\UU$,  there is a unique component $U^1\in\UU^1$ parallel to $U$ rel $\PPP_f$ and $f(U^1)$ is a component of $\,\UU$. Hence $f_{\G}$ is well-defined on $\UU$. Since $\UU$ is a finite collection, every $U\in\UU$ is eventually $f_{\G}$-periodic. In other words, If $U\in\UU$ is periodic with period $p\ge 1$, then there is a unique complex type component $U^p\in\UU^p$ parallel to $U$ rel $\PPP_f$ such that $f^p(U^p)=U$. Theorem \ref{thm:combinatorial} is proved by the following steps. 

\vskip 0.24cm
\noindent {\bf Step 1: Deform $f$ to a standard form $F$. }

A multi-annulus is a finite collection of annuli in $\cbar\sm\PPP_f$ with finite modulus such that they are pairwise disjoint and pairwise non-isotopic rel $\PPP_f$. Choose a multi-annulus $\AAA\subset\cbar\sm\PPP_f$ isotopic to $\G$ rel $\PPP_f$ such that its boundary $\pa \AAA$ is a disjoint union of Jordan curves in $\cbar\sm\PPP_f$. Let $\AAA^1$ be the union of all the essential components (which are isotopic to curves in $\G$ rel $\PPP_f$) of $f^{-1}(\AAA)$. Because $\G$ is pre-stable, so for each $\g\in\G$, there is at least one component of $\AAA^1$ isotopic to $\g$ rel $\PPP_f$. 

For each $\g\in\G$, denote by $A^*({\g})$ the smallest annulus containing all the components of $\AAA^1$ which are isotopic to $\g$ rel $\PPP_f$. Then its boundary are two Jordan curves in $\cbar\sm\PPP_f$ isotopic to $\g$ rel $\PPP_f$. Set $\AAA^*(\G)=\cup_{\g\in\G}A^*({\g})$. Then there exists a neighborhood $\NNN$ of $\PPP_f$ and a homeomorphism $\theta$ of $\cbar$ such that $\theta$ is isotopic to the identity rel $\NNN$ and $\theta(\AAA)=\AAA^*(\G)$.  Set $F=f\circ\theta$. Note that $\PPP_F=\PPP_f$. Then each essential component of $F^{-1}(\AAA)$ is contained in $\AAA$ and $\pa\AAA\subset\pa F^{-1}(\AAA)$. 
Such $F$ is called a {\bf standard form} of $f$ with respect to $\G$. 

\vskip 0.24cm 
\noindent {\bf Step 2: Define a new branched covering for each periodic component $U$. }

Suppose that $U\in\UU$ is a periodic component of $f_{\G}$ with period $p\geq 1$. Then there is a unique component $W$ of $\cbar\sm\AAA$ isotopic to $U$ rel $\PPP_f$ and $F^{-p}(W)$ has a unique component $W^p\subset W$ such that $\pa W\subset\pa W^p$. Let $\theta_1$ be an isotopy from $U$ to $W$ and $\theta_2$ be an isotopy from $U^p$ to $W^p$ such that $\theta_1\circ f^p=F^p\circ\theta_2$. For each component $\g$ of $\partial W^p$, denote by $D(\g)$ the component of $\cbar\sm\g$ disjoint from $W^p$. For each curve $\g\in\pa W$, mark one point $a_{\g}\in D(\g)\cap (\cbar\sm U)$ and denote by $\PPP_U$ the set of these marked points. 

For each component $\alpha$ of $\partial W^p$, $\g:=F^p(\alpha)$ is a component of $\partial W$. There is a branched covering $H_{\alpha}:\,D(\alpha)\to D(\g)$ such that
\begin{itemize}
\item $H_{\alpha}$ can be extended continuously to the boundary and $H_{\alpha}=F^p$ on $\alpha$,
\item $H_{\alpha}$ has exactly one critical value at $a_{\g}$ if $\deg(F^p:\,\alpha\to\g)>1$,
\item if $\alpha$ is a curve of $\partial W$, then $H_{\alpha}(a_{\alpha})=a_{\g}$ and $H_{\alpha}$ is holomorphic in a neighborhood of $a_{\alpha}$.
\end{itemize}
Define a branched covering $G:\,\cbar\to\cbar$ by:
$$
G=
\begin{cases}
F^p & \text{on } \ol{W^p}, \\
H_{\alpha} & \text{on } D(\alpha) \text{ for each component }\alpha \text{ of } \pa W^p. 
\end{cases}
$$
Then $G$ is a branched covering. Denote by $\wt\PPP=\PPP_U\cup(\PPP_f\cap W)$. Then it is a finite set containing $\PPP_G$ satisfying $G(\wt\PPP)\subset\wt\PPP$. Hence $(G,\wt\PPP)$ is a marked branched covering.

\vskip 0.24cm 
\noindent {\bf Step 3: Get the new rational map by Thurston theorem. }

For any stable multicurve $\G_1$ of $(G,\wt\PPP)$, one may assume that each curve in $\G_1$ is contained in $W$ since each component of $\cbar\sm W$ contains at most one point of $\wt\PPP$. Let $M_1$ and $M_2$ be the transition matrices of $\G_1$ under $F^p$ and $(G,\wt\PPP)$, respectively. Then $M_1\ge M_2$ and hence $\lambda(M_1)\ge\lambda(M_2)$. Because $F^p$ is Thurston equivalent to $f^p$ through $\theta_1$ and $\theta_2$, so $F^p$ has no Thurston obstructions and then $\lambda(M_1)<1$. Hence $(G,\wt\PPP)$ has no Thurston obstructions. 

By Theorem \ref{thm:thurston}, the marked branched covering $(G,\wt\PPP)$ is Thurston equivalent to a marked rational map $(g,\PPP)$, i.e., there exist orientation preserving homeomorphisms $(\phi_0,\phi_1)$ of $\cbar$ such that $\phi_0$ is isotopic to $\phi_1$ rel $\wt\PPP$ and $\phi_0\circ G=g\circ\phi_1$ on $\cbar$. Since $G$ is holomorphic in a neighborhood of $\wt{\PPP}$, one may choose $\phi_0$ such that both $\phi_0$ and $\phi_1$ are holomorphic in a neighborhood of $\wt{\PPP}$.

\vskip 0.24cm
\noindent {\bf Step 4: Relate $g$ to $f$}

Since $G(\PPP_U)\subset\PPP_U$, each point in $\PPP_U$ is eventually periodic. Moreover, each cycle in $\PPP_U$ contains critical points of $G$, for otherwise it will induce a Levy cycle of $F^p$ and it contradict that $F^p$ has no Thurston obstructions. Denote $\PPP_0=\phi_0(\PPP_U)$. Then $g(\PPP_0)\subset\PPP_0$ and each cycle in $\PPP_0$ is super-attracting. Since each component of $\cbar\sm\phi_0(\overline{W})$ contains exactly one point of $\PPP_0$ and disjoint from $\PPP\sm\PPP_0$, $\phi_0$ can be chosen such that $\cbar\sm\phi_0(\overline{W})\subset\FFF_{g}$, and each component of $\cbar\sm\phi_0(\overline{W})$ is a round disk in B\"{o}ttcher coordinates of the super-attracting domains of $g$. Thus $\phi_1(W^p)=g^{-1}(\phi_0(W))$ is compactly contained in $\phi_0(W)$.

Note that $G=F^p=\theta_1\circ f^p\circ\theta_2^{-1}$ on $W^p$ and $\theta_2^{-1}(W^p)=U^p$. Thus
$$
\phi_0\circ G=\phi_0\circ\theta_1\circ f^p\circ\theta_2^{-1}=g\circ\phi_1\,\text{ on }W^p.
$$
Set $\varphi_0=\phi_0\circ\theta_1$ and $\varphi_1=\phi_1\circ\theta_2$. Then $\varphi_1(U^p)=\phi_1(\theta_2(U^p))=\phi_1(W^p)$ is compactly contained in $\varphi_0(U)=\phi_0(W)$, and $\varphi_0\circ f^p=g\circ\varphi_1\,\text{ on }U^p$. 
$$
\diagram
W^p\drto_{\phi_1} & U^p\dto^{\varphi_1}\lto_{\theta_2}\rto^{f^p} & U\rto^{\theta_1} \dto^{\varphi_0} & W\dlto^{\phi_0}\\
& \phi_1(W^p)=\varphi_1(U^p)\rto^{g} & \varphi_0(U)=\phi_0(W)
\enddiagram
$$

Because each component of $\cbar\sm W$ contains one marked point, the Thurston equivalence class of $g$ is uniquely determined. So the rational map $g$ is unique up to holomorphic conjugacy by Theorem \ref{thm:thurston}. Now the proof of statements (1) is complete.

\vskip 0.24cm
\noindent{\bf Step 5: Put $\JJJ_g$ into $\JJJ_f$. }

For the proof of statement (2), we need to recover a PCF branched covering from the rational map $g$ such that it is Thurston equivalent to $f^p$. Denote $V=\phi_0(W)$ and $V^p=\phi_1(W^p)$. Then $V^p$ is compactly contained in $V$. We may assume that for any component $\alpha$ of $\pa W^p\sm\pa W$, $\phi_0(\alpha)$ is compactly contained in $D(\phi_1(\alpha))$. It is clear that $\JJJ_g$ is compactly contained in $V^p$. 

For each component $\alpha$ of $\pa W^p$, $\phi_0^{-1}\circ\phi_1(\alpha)$ is compactly contained in $W^p$ and is isotopic to $\alpha$ rel $\PPP_f$. Then $\phi_1\circ\phi_0^{-1}\circ\phi_1(\alpha)$ is compactly contained in $V^p$ and is isotopic to $\phi_0(\alpha)$ rel $\PPP$. 

Denote by $\wt{W^p}$ the subset of $W^p$ bounded by $\phi_0^{-1}\circ\phi_1(\alpha)$, where $\alpha\in\pa W^p$. Then $\wt{W^p}\cap \PPP_f=W^p\cap \PPP_f$ and $\JJJ_g$ is compactly contained in $\phi_1(\wt{W^p})$. Since $\phi_1$ is isotopic to $\phi_0$ rel $\wt\PPP$, there is a homeomorphism $\psi_0:\, W\to V$ (see Figure \ref{fig:Jg.pdf}) such that
\begin{itemize}
\item $\psi_0=\phi_1$ on $\wt{W^p}$, 
\item $\psi_0$ is isotopic to $\phi_0|_{W}$ rel $\pa W\cup(W\cap\PPP_f)$. 
\end{itemize}

\begin{figure}[htbp]
\centering
\includegraphics[width=12cm]{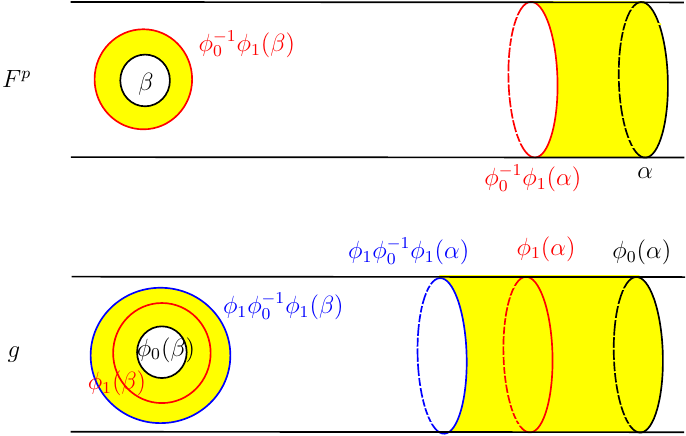}
\caption{The components $\beta$ of $\pa W^p\sm\pa W$ and $\alpha$ of $\pa W(\subset\pa W^p)$ are marked. The homeomorphism $\psi_0$ maps the yellow annuli for $F^p$ to the yellow annuli for $g$.}
\label{fig:Jg.pdf}
\end{figure}

Define
\begin{equation*}
\wt{F}=
\begin{cases}
\psi_0^{-1}\circ g\circ \phi_1 & \text{on }W^p\\
F^p & \text{on }\cbar\sm W^p.
\end{cases}
\end{equation*}
Recall that $\phi_0^{-1}\circ g\circ\phi_1=G=F^p$ on $\pa W^p$. So $\wt{F}=F^p$ on $\partial W^p$ and hence is a branched covering of $\cbar$. It is easy to check that $\PPP_{\wt{F}}=\PPP_{F}=\PPP_f$. Define
\begin{equation*}
\theta_3=
\begin{cases}
\psi_0^{-1}\circ\phi_0 & \text{on } W^p \\
\text{id} & \text{on }\cbar\sm W^p.
\end{cases}
\end{equation*}
Then $\theta_3$ is a homeomorphism isotopic to the identity relative to a neighborhood of $\PPP_f$ and $\wt{F}=\theta_3\circ F^p$ on $\cbar$. Consequently, $\wt{F}=\theta_3\circ\theta_1\circ f^p\circ\theta_2^{-1}$.

All of the three homeomorphisms $\theta_1,\theta_2$ and $\theta_3$ are isotopic to the identity relative to a neighborhood of $\PPP_f$. Thus $\wt{F}$ is Thurston equivalent to $f^p$ and satisfies the condition of Theorem \ref{thm:semiconjugacy}. Hence there is a semi-conjugacy $h$ from $\wt{F}$ to $f^p$.

Since $\psi_0=\phi_1$ on $\wt{W^p}$, we have $\wt{F}=\phi_1^{-1}\circ g\circ\phi_1$ in a neighborhood of $\phi_1^{-1}(\JJJ_g)$. For any point $z\in h^{-1}(\FFF_f)$, its forward orbit $\{\wt{F}^n(z)\}$ converges to a super-attracting cycle in $\PPP_f$. On the other hand, for any point $z\in\phi_1^{-1}(\JJJ_g)$, its forward orbit $\{\wt{F}^n(z)\}$ is always contained in $\phi_1^{-1}(\JJJ_g)$, which is disjoint from super-attracting cycles in $\PPP_f$. It follows that $\phi_1^{-1}(\JJJ_g)\subset h^{-1}(\JJJ_f)$. Set $\pi=h\circ\phi_1^{-1}$. Then $\pi(\JJJ_g)\subset\JJJ_f$ and $f^p\circ\pi=\pi\circ g$ on $\JJJ_g$. Now the proof of statement (2) is complete. \hfill$\square$

\section{Coiling curves separate small Julia sets}
In this section, we will prove Theorem \ref{thm:disjoint}. Before the proof, we give some dynamical properties on curves. 

\subsection{Dynamics on curves}\label{sec:dynamicsoncurves}
Let $(F,\PPP)$ be a marked branched covering and $\G$ be a completely stable multicurve of $(F,\PPP)$. A curve $\g\in\G$ is {\bf periodic} if there exists an integer $p\ge 1$ such that $F^{-p}(\g)$ contains a curve isotopic to $\g$ rel $\PPP$. For a periodic curve $\g\in\G$, if the collection $\OO_{\g}:=\{\g_0=\g,\g_1,\cdots,\g_{p-1}\}$ satisfies that 
\begin{itemize}
\item $f^{-1}(\g_{i})$ has a component isotopic to $\g_{i-1}$ rel $\PPP$ (set $\g_p=\g_0=\g$);
\item $\g_i$ is not isotopic to $\g_j$ rel $\PPP$ for any $i\neq j$, 
\end{itemize}
then $\OO_{\g}$ is a called a {\bf cycle} containing $\g$. A curve $\g\in\G$ is {\bf pre-periodic} if $\g$ is isotopic to some component of $F^{-k}(\alpha)$ for some periodic curve $\alpha$ and some integer $k\geq 1$. Each periodic curve $\g\in\G$ {\bf generates a completely stable multicurve $\mathbf{\Lambda_{\g}}\subset\G$} as the following: $\beta\in\Lambda_{\g}$ if there is an integer $k\ge 1$ such that $F^{-k}(\g)$ contains a curve isotopic to $\beta$ rel $\PPP$. 

\begin{lem}\label{lem:pre-periodic}
Each curve in a pre-stable multicurve is pre-periodic. 
\end{lem}
\begin{proof}
Let $\G$ be a pre-stable multicurve. Then for each curve $\g\in\G$, there is a curve $\g_1\in\G$ such that $\g$ is isotopic to some component of $F^{-1}(\g_1)$. Inductively, there is a sequence $\{\g_i\}$ in $\G$ with $\g_0=\g$ such that $F^{-1}(\g_{i+1})$ contains a curve isotopic to $\g_i$ rel $\PPP$ for $i\ge 0$. Because $\G$ contains finitely many curves, there are integers $0\le i_1\le i_2\le\#\G$ such that $\g_{i_1}$ is isotopic to $\g_{i_2}$ rel $\PPP$, i.e., $\g_{i_1}$ is periodic. Hence $\g$ is pre-periodic. 
\end{proof}

Recall that for $\g\in\G$, $\kappa_n(\g,\G)$ is the number of curves in $F^{-n}(\G)$ isotopic to $\g$ rel $\PPP$. The curve $\g$ is called coiling in $\G$ if $\kappa_n(\g,\G)$ tends to infinity as $n$ tends to infinity. Note that by definition $\kappa_{n+1}(\alpha,\G)\geq\kappa_n(\g,\G)$ for each essential curve $\alpha$ in $F^{-1}(\g)$. So if $\g$ is a coiling curve of $\G$, then each curve of $\Lambda_{\g}$ is a coiling curve of $\,\G$. 

\begin{lem}\label{lem:periodcoiling}
Let $\OO_{\alpha}$ and $\OO_{\g}$ be two cycles in $\,\G$ containing $\alpha$ and ${\g}$ respectively such that $\g\notin\OO_{\alpha}$ and some component of $F^{-m}(\alpha)$ is isotopic to a curve in $\OO_{\g}$ rel $\PPP$ for some integer $m\geq 1$. 
Then $\g$ is a coiling curve of $\,\G$. 
\end{lem}

\begin{proof}
Note that if a curve in a cycle is coiling in $\G$, then every curve in the cycle is coiling in $\G$. So we only need to prove that a curve in $\OO_{\g}$ is coiling in $\G$. Without loss of generality, we may assume that $\OO_{\g}$ is the cycle such that the number $m$ is the smallest integer satisfying the condition in the lemma. 

\begin{figure}[htbp]
\centering
\includegraphics[width=8cm]{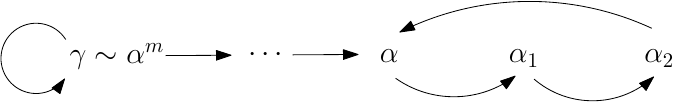}
\caption{A cycle $\{\alpha,\alpha_1,\alpha_2\}$ not containing $\g$.}
\label{fig:relation}
\end{figure}
We may also assume that the curve in $\OO_{\g}$ which is isotopic to a component of $F^{-m}(\alpha)$ rel $\PPP$ is $\g$ (see Figure \ref{fig:relation}). Since $\g$ is periodic, there exists an integer $p\geq 1$ such that some component $\g^p$ of $F^{-p}(\g)$ is isotopic to $\g$ rel $\PPP$. Then a component $\alpha^{m+p}$ of $F^{-p}(\alpha^m)$ is isotopic to $\g$ rel $\PPP$. We prove by contradiction that 
$$
\alpha^{m+p}\cap\alpha^m=\emptyset.
$$ 
Let $q\geq 1$ be the period of $\OO_{\alpha}$. By replacing $\alpha$ by some component of $F^{-lq}(\alpha)$ isotopic to $\alpha$ rel $\PPP$ if necessary for some $l\geq 1$, we assume $F^p(\alpha^m)$ is a Jordan curve. If $\alpha^{m+p}\cap\alpha^m\neq\emptyset$, then $\alpha^m\cap F^p(\alpha^m)=F^p(\alpha^{m+p})\cap F^p(\alpha^m)\supset F^p(\alpha^{m+p}\cap\alpha^m)\neq\emptyset$. On the other hand, $F^p(\alpha^m)$ is either not periodic or lies in the cycle not containing $\g$. This contradicts $\alpha^m\cap F^p(\alpha^m)\neq\emptyset$. 

Denote by $\alpha^{m+kp}$ a component of $F^{-kp}(\alpha^m)$ isotopic to $\g$ rel $\PPP$. Then by the same argument $
\alpha^{m+k_1p}\cap\alpha^{m+k_2p}=\emptyset \text{ if } k_1\neq k_2$. Hence $\kappa_{m+kp}(\g,\G)\geq k\to\infty$ as $k\to\infty$. In other words, $\g$ is a coiling curve of $\,\G$. 
\end{proof}

\begin{lem}\label{lem:unique}
Let $\g\in\G$ be a periodic curve and $\Lambda_{\g}\subset \G$ is the completely stable multicurve generated by $\g$. If $\g$ is not coiling in $\G$, then $\kappa_n(\g,\G)=\kappa_n(\g,\Lambda_{\g})=1$ for all $n\geq 1$. 
\end{lem}
\begin{proof}
We first claim that $\kappa_n(\g,\G)=\kappa_n(\g,\Lambda_{\g})$. In other words, for any curve $\alpha\in\G\sm\Lambda_{\g}$ and any integer $n\geq 0$, any component of $F^{-n}(\alpha)$ is not isotopic to $\g$ rel $\PPP$. Let $\alpha\in\G\sm\Lambda_{\g}$ be a curve which is not periodic. Then there is a periodic curve $\beta\in\G$ such that some components of $F^{-m}(\beta)$ is isotopic to $\alpha$ rel $\PPP$ by Lemma \ref{lem:pre-periodic}. If some component of $F^{-k}(\alpha)$ is isotopic to $\g$ rel $\PPP$ for some $k\geq 1$, then some component of $F^{-m-k}(\beta)$ will be isotopic to $\g$ rel $\PPP$. Therefore if the claim is true for all periodic curve in $\,\G\sm\Lambda_{\g}$, then it is true for all curves in $\,\G\sm\Lambda_{\g}$. Assume that the claim is not true for a periodic curve $\alpha\in\G\sm\Lambda_{\g}$. Then by Lemma \ref{lem:periodcoiling}, $\g$ will be a coiling curve of $\G$. This is a contradiction. So $\kappa_n(\g,\G)=\kappa_n(\g,\Lambda_{\g})$. 

Now we prove that $\kappa_n(\g,\Lambda_{\g})=1$. Since $\g$ is not coiling in $\G$, for any $n\geq 1$, $F^{-n}(\g)$ has at most one curve isotopic to $\g$ rel $\PPP$. In fact, if there is an integer $n_0\geq 1$ such that $F^{-n_0}(\g)$ has $k\geq 2$ curves isotopic to $\g$ rel $\PPP$, then $\kappa_{m\cdot n_0}(\g,\G)\geq k^m\to\infty$ as $m\to\infty$. Note that $\kappa_{n+1}(\g,\G)\geq\kappa_n(\g,\G)$ since $\G$ is pre-stable. Hence $\g$ is coiling in $\G$. This is a contradiction. By the definition of $\Lambda_{\g}$, there is an integer $N>0$, such that the isotopy classes of $\Lambda_{\g}$ equals the isotopy classes of essential curves in $F^{-N}(\g)$. So $1\leq\kappa_n(\g,\Lambda_{\g})\leq\kappa_{n+N}(\g,\{\g\})=1$ for all $n\geq 1$.  
\end{proof}

\begin{lem}\label{lem:aperiodic-coiling}
Suppose $\G$ is completely stable. If $\g$ is a coiling curve of $\,\G$ but it is not periodic, then there exists a periodic coiling curve $\alpha$ and an integer $k\geq 1$ such that $\g$ is isotopic to some component of $F^{-k}(\alpha)$ rel $\PPP$. 
\end{lem}
\begin{proof}
Let $\{\alpha_i\}_{i=1}^m$ be the collection of all the periodic curves in $\G$ such that some components $F^{-k_i}(\alpha_i)$ are isotopic to $\g$ rel $\PPP$ for some $k_i\geq 1$. By Lemma \ref{lem:pre-periodic}, $\{\alpha_i\}_{i=1}^m\neq\emptyset$. Suppose that each $\alpha_i$ is not coiling, i.e. $\kappa_n(\alpha_i,\G)$ is uniformly bounded for all $n\geq 1$. Then 
$$
\sum_{i=1}^m l_i\cdot\kappa_{n-k_i}(\alpha_i,\G)=\kappa_{n}(\g,\G)
$$
is uniformly bounded for all $n\geq 1$, where $l_i=\kappa_{k_i}(\g,\{\alpha_i\})\geq 1$. This contradicts that $\g$ is a coiling curve of $\G$. 
\end{proof}

\subsection{Proof of Theorem \ref{thm:disjoint}}\label{sec:proof1.2}
Let $f$ be a PCF rational map and $\G$ be a completely stable multicurve. Now suppose that $\g\in\G$ is not coiling in $\G$. Then $\g$ is either periodic with $\kappa_n(\g,\G)=1$ for all $n\geq 1$ by Lemma \ref{lem:unique} or $\g$ is not periodic. If $\g$ is not periodic, then by Lemma \ref{lem:pre-periodic} there is a periodic curve $\alpha$ such that some component of $f^{-m}(\alpha)$ is isotopic to $\g$ rel $\PPP_f$ for some integer $m\geq 1$. Let $\Delta_{\g}$ be the collection of all such $\alpha$. For $\alpha\in\Delta_{\g}$, set 
$$
N_{\g\alpha}:=\min\{m\geq 1: \text{ some component of }f^{-m}(\alpha)\text{ is isotopic to }\g \text{ rel }\PPP_f\}. 
$$
We denote 
$$
N:=\max\{N_{\g\alpha}: \g\in\G \text{ is not coiling in }\G\text{ and } \alpha\in\Delta_{\g}\}. 
$$
Set $\G_N$ the collection of essential curves in $f^{-N}(\G)$. Add a marked point in each annulus bounded by two isotopic curves rel $\PPP_f$ in $\G_N$ and denote by $\PPP_*$ the set of these marked points. We may choose these marked points suitably such that $f(\PPP_*)\subset\PPP_*\cup\PPP_f$. Denote by $\PPP=\PPP_f\cup\PPP_*. $ Then $\G_N$ is a completely stable multicurve of $(f,\PPP)$. 

Because $\g$ is not coiling in $\G$, every $\alpha\in\Delta_{\g}$ is not coiling. So by Lemma \ref{lem:unique}, $\kappa_n(\alpha,\G)=1$ for all $n\geq 1$. Thus for all $\g\in\G$ which is not coiling in $\G$, $\kappa_n(\g,\G)=\kappa_N(\g,\G)$ for all $n\geq N$. Therefore for $\g\in\G_N$, either $\kappa_n(\g,\G_N)=1$ for all $n\geq 1$ or $\kappa_n(\g,\G_N)\to\infty$ as $n\to\infty$. In the following, we will prove that $\PPP$ and $\G_{\PPP}:=\G_N$ satisfy Theorem \ref{thm:disjoint}. 

\subsubsection{Sufficiency}\label{sec:sufficiency}
In this section, we will prove that if $\g$ is a coiling curve in $\G_{\PPP}$, then $\JJJ_1\cap\JJJ_2=\emptyset$. The strategy of the proof follows that in \cite{cui2016renormalizations}: Deform (some iteration of) the marked rational map $(f,\PPP)$ to a branched covering $(F,\PPP)$ in its Thurston equivalence class. Applying Theorem \ref{thm:semiconjugacy}, we obtain a semi-conjugacy from $F$ to (some iteration of) $f$. Denote by $A(\g)$ the annulus isotopic to the curve $\g$ rel $\PPP$. The key point is to prove that if $\g$ is coiling in $\,\G$,  then any two points in different components of $\cbar\sm A(\g)$ will not map to one point under the semi-conjugacy. This implies that $\JJJ_{1}\cap\JJJ_{2}=\emptyset$. 

\vskip 0.24cm
\begin{proof}[Proof of the sufficiency of Theorem \ref{thm:disjoint}]
Denote by $\UU:=\UU_{\G_{\PPP}}$ the collection of components of $\cbar\sm\G_{\PPP}$. Then each component of $\UU$ is pre-periodic under $f_{\G_{\PPP}}$. Let $m$ be the least common multiple of the periods of periodic components of $\UU$.  Then each periodic component of $\cbar\sm\G_{\PPP}$ is fixed under $f^m_{\G_{\PPP}}$. 

Denote by $\AAA_{\PPP}$ the multi-annulus isotopic to $\G_{\PPP}$ rel $\PPP$. As in \S \ref{sec:combinatorial}, denote by $F_0$ the standard form of $f$ corresponding to the multicurve $\G_{\PPP}$. Then for each fixed $U_i$ of $f^m_{\G_{\PPP}}$, we can associate it a $W_i,W_i^1$ and define a branched covering $\psi_{0,i}^{-1}\circ g_i\circ\phi_{1,i}:W_i^1\to W_i$ (refer to Step 5 of \S \ref{sec:combinatorial}). We define a marked branched covering $F$ of $\cbar$ as follows: 
\begin{itemize}
\item On each component $W_i^1$, we define $F$ as $\psi_{0,i}^{-1}\circ g_i\circ \phi_{1,i}$. 
\item On the complement of the union of these $W_i^1$'s, we define $F$ as $F_0^m$. 
\end{itemize}
Then $F$ is Thurston equivalent to $f^m$ rel $\PPP$. Moreover, essential components of $F^{-1}(\AAA_{\PPP})$ are contained in $\AAA_{\PPP}$ and $\pa\AAA_{\PPP}\subset\pa F^{-1}(\AAA_{\PPP})$. By Theorem \ref{thm:semiconjugacy}, there is a semi-conjugacy $h:\cbar\to\cbar$ from $F$ to $f^m$. What's more, $h$ is holomorphic in a neighborhood of attracting periodic points of $F$.  

If $\g$ is a coiling curve, then there exists a completely stable multicurve $\G_{\g}\subset\G_{\PPP}$ such that each curve in $\G_{\g}$ is a coiling curve of $\G_{\PPP}$. In fact, when $\g$ is periodic, take $\Lambda_{\g}$ as $\G_{\g}$. When $\g$ is not periodic, by Lemma \ref{lem:aperiodic-coiling}, there exists periodic coiling curve $\alpha\in\G_{\PPP}$ and an integer $k\geq 1$ such that $\g$ is isotopic to some component of $f^{-k}(\alpha)$ rel $\PPP$. So take $\Lambda_{\alpha}$ as $\G_{\g}$. Denote 
$$
\AAA_0=\bigcup_{\beta\in\G_{\g}} A({\beta}), 
$$
where $A(\beta)$ is the annulus in $\AAA_{\PPP}$ isotopic to $\beta$ rel $\PPP$. 

\begin{lem}\label{lem:key}
For any integer $n\geq 0$ and any two distinct components $E_1$ and $E_2$ of $\cbar\sm F^{-n}(\AAA_0)$, $h(E_1)\cap h(E_2)=\emptyset$.
\end{lem}

If Lemma \ref{lem:key} is true, then $\JJJ_1\cap\JJJ_2=\emptyset$ since $h^{-1}(\JJJ_1)$ and $h^{-1}(\JJJ_2)$ are separated by $A(\g)$ which is a component of $\AAA_0$. The proof of the sufficiency of Theorem \ref{thm:disjoint} is complete. 
\end{proof}

\vskip 0.24cm
The proof of Lemma \ref{lem:key} based on the following lemma \cite{cui2016renormalizations}. We include their proof here for self-containedness.
	
\begin{lem}\label{lem:top}
Let $\G\subset\cbar$ be a finite disjoint union of Jordan curves. Suppose that $D\subset\cbar$ is a Jordan domain and $E\subset D$ is a continuum. Then there is an integer $M<\infty$ such that for any two distinct points $z_1,z_2\in E$, there exists a Jordan arc $\delta\subset D$ connecting $z_1$ with $z_2$ such that $\#(\delta\cap\G)\le M$.
\end{lem}

\begin{proof}
Denote
\begin{eqnarray*}
\mathscr{L} &=& \{\sigma: \sigma\text{ is a component of }\G\cap D\}\text{ and} \\
\mathscr{L}_0 &=& \{\sigma\in\mathscr{L}:\,\sigma\cap E\neq\emptyset\}.
\end{eqnarray*}
Let $\tau: S^1\times\{1,2,\cdots,n\}\to\G$ be a homeomorphism. Then $\tau^{-1}(\G\cap E)$ is a compact subset, which is covered by the open sets $\{\tau^{-1}(\sigma)\}$ ($\sigma\in\mathscr{L}$) in $S^1\times\{1,2,\cdots,n\}$. Thus there are finitely many such open sets cover $\tau^{-1}(\G\cap E)$. So $M:=\#\mathscr{L}_0<\infty$.

For any two distinct points $z_1,z_2\in E$, denote by $\mathscr{L}(z_1,z_2)$ the collection of arcs $\sigma\in\mathscr{L}$ such that either $\sigma\cap\{z_1,z_2\}\neq\emptyset$ or $\sigma\cup\partial D$ separates $z_1$ from $z_2$. Then for each $\sigma\in\mathscr{L}(z_1,z_2)$, $\sigma\cap E\neq\emptyset$ since $E$ is connected. Thus $\mathscr{L}(z_1,z_2)\subset\mathscr{L}_0$ and hence $\#\mathscr{L}(z_1,z_2)\le\#\mathscr{L}_0=M<\infty$.

By definition of $\mathscr{L}(z_1,z_2)$, there exists a Jordan arc $\delta\subset D$ connecting $z_1$ with $z_2$ such that $\delta$ intersects each $\sigma\in\mathscr{L}(z_1,z_2)$ at a single point and is disjoint from any arcs in $\mathscr{L}\sm\mathscr{L}(z_1,z_2)$. So $\#(\delta\cap\G)\le\#\mathscr{L}(z_1,z_2)\le M$.
\end{proof}

We say that a continuum $E$ crosses an annulus $A$ if $E$ intersects both of the two boundary components of $A$. Note that for any $z\in\cbar$, $h^{-1}(z)$ is a full continuum by Theorem \ref{thm:semiconjugacy}(4). 

\begin{lem}\label{lem:empty}
The set $$T=\{z\in\cbar: h^{-1}(z)\text{ crosses a component of }\AAA_0\}$$ is an empty set.
\end{lem}

\begin{proof}
By Theorem \ref{thm:semiconjugacy}(4), $T\subset\JJJ_f$. Moreover, $T$ is a closed set. In fact, let $\{z_n\}_{n\ge 1}$ be a sequence of points in $T$ and converges to a point $z_{\infty}\in\cbar$. Passing to a subsequence, we may assume that for all $n\ge 1$, there exist two points $\zeta_n,\zeta'_n\in h^{-1}(z_n)$, such that they are contained in the two distinct boundary components of the same annulus $A\in\AAA_0$. By taking subsequence furthermore, we may further assume that $\{\zeta_n\}_{n\ge 1}$ and $\{\zeta'_n\}_{n\ge 1}$ converge to the points $\zeta_{\infty}$ and $\zeta'_{\infty}$ respectively. The continuity of $h$ implies that $h(\zeta_{\infty})=z_{\infty}=h(\zeta'_{\infty})$. Thus $z_{\infty}\in T$ and \,$T$ is a closed set. 

For each point $z\in T$, $h^{-1}(z)$ crosses some component $A$ of $\AAA_0$. Since $\G_{\g}$ is a completely stable multicurve, $F^{-1}(\AAA_0)$ has a component $A^1$ contained in $A$ essentially. By Theorem \ref{thm:semiconjugacy}(5), $F(h^{-1}(z))=h^{-1}(f(z))$. So $h^{-1}(f(z))$ crosses $F(A^1)$ which is a component of $\AAA_0$. So $f(z)\in T$. In other words, $f(T)\subset T$. 

Assume that $T$ is non-empty. Set $T_{\infty}=\bigcap_{n\geq 1}f^{n}(T)$. Then $T_{\infty}$ is a non-empty closed set and $f(T_{\infty})=T_{\infty}$. Hence for any point $z_0 \in T_{\infty}$, there is a sequence of points $\{z_n\}_{n\ge 0}$ contained in $T_{\infty}$ such that $f(z_{n+1})=z_n$ for any $n\geq 0$ (i.e. $T_{\infty}$ contains a backward orbit). Either $z_n$ is periodic for all $n\geq 0$ or there is an integer $n_0$ such that $z_{n}$ is not periodic for any $n \geq n_0$. In the former case, all the points $z_n$ are not critical points of $f$ since $z_n\in\JJJ_f$. In the latter case, there is an integer $n_1\geq 0$ such that $z_n$ are non-critical points of $f$ for all $n\geq n_1$ since $f$ has only finitely many critical points. So in both cases, we obtained a sequence of points $\{z_n\}_{n\geq 0}$ in $T_{\infty}\sm \CCC_f$ such that $f(z_{n+1})=z_n$, where $\CCC_f$ is the set of critical points of $f$. 

Set $L_n = h^{-1}(z_n)$. Then by Theorem \ref{thm:semiconjugacy}(5), $L_n$ is a component of $F^{-1}(L_{n-1})$ and there is a Jordan domain $D_0$ containing $L_0$ such that $F^{n}: D_n\to D_0$ is a homeomorphism for any $n\ge 1$, where $D_n$ is the component of $F^{-n}(D_0)$ containing $L_n$. By Lemma \ref{lem:top}, there exists an integer $M\ge 0$ such that for any two distinct points $\zeta_0,\zeta'_0\in L_0$, there is a Jordan arc $\delta\subset D_0$ connecting $\zeta_0$ with $\zeta'_0$ and $\#(\delta\cap\G_{\PPP})\le M$.

Note that for any curve $\beta\in\G_{\g}$, $\kappa_n(\beta,\G_{\PPP})\to\infty$. So there exists an integer $n>0$ such that for any component $A$ of $\AAA_0$, there are at least $M+1$ components of $F^{-n}(\AAA_{\PPP})$ contained essentially in $A$. Since $L_n$ crosses a component of $\AAA_0$, there exist two distinct points $\zeta_n,\zeta'_n\in L_n$ such that $F^{-n}(\G_{\PPP})$ has at least $M+1$ components separating $\zeta_n$ from $\zeta'_n$.

Now the two points $F^{n}(\zeta_n)$ and $F^{n}(\zeta'_n)$ are contained in $L_0$. We have known that there is a Jordan arc $\delta\subset D_0$ such that $\delta$ connects $F^{n}(\zeta_n)$ with $F^{n}(\zeta'_n)$ and $\#(\delta\cap\G_{\PPP})\le M$. Let $\delta_n$ be the component of $F^{-n}(\delta)$ contained in $D_n$. Then $\delta_n$ connects $\zeta_n$ with $\zeta'_n$. Since $F^{n}:\delta_n\to\delta$ is a homeomorphism, we have $\#(\delta_n\cap F^{-n}(\G_{\PPP}))\le M$. This contradicts the fact that $F^{-n}(\G_{\PPP})$ has at least $M+1$ essential components separating $\zeta_n$ from $\zeta'_n$. Hence $T$ is empty.
\end{proof}

\noindent {\bf Remark:} The statement of Lemma \ref{lem:empty} is the same as Claim 1 in \cite[\S 4]{cui2016renormalizations}. However, the proof here is slightly different from there. In \cite{cui2016renormalizations}, $\AAA_0=\AAA_{\PPP}$ since each curve in a Cantor multicurve is a coiling curve. So to prove the set $T$ is empty, they can use $\AAA_{\PPP}$ to get the contradiction. In this paper, $\G_{\PPP}$ may not contain Cantor multicurves and in this case we are not able to find a submulticurve $\G'$ of $\G$ such that each curve in $\G'$ is coiling in $\G'$. So we need two different multicurves $\G_{\g}$ and $\G_{\PPP}$ to deduce the contradiction.  The two multicurves correspond two collections of annuli $\AAA_0$ and $\AAA_{\PPP}$ respectively. The set $T$ contains the points whose fiber crosses annuli in $\AAA_0$ and the annulus in $\AAA_{\PPP}$ gives the increasing number of annuli which are contained in $\AAA_0$. 
\vskip 0.24cm

Now we are ready to prove Lemma \ref{lem:key}. 
\begin{proof}[Proof of Lemma \ref{lem:key}]
Suppose $E_1$ and $E_2$ are two different components of $\cbar\sm F^{-n}(\AAA_0)$ and $h(E_1)\cap h(E_2)\neq\emptyset$. Take a point $z\in h(E_1)\cap h(E_2)$. Since $h^{-1}(z)$ is a continuum and both $E_1$ and $E_2$ have intersection with $h^{-1}(z)$, $h^{-1}(z)$ crosses a component $A$ of $F^{-n}(\AAA_0)$. By Theorem \ref{thm:semiconjugacy}(5), $h^{-1}(f^n(z))=F^n(h^{-1}(z))$ crosses $F^n(A)$ which is a component of $\AAA_0$. This contradicts Lemma \ref{lem:empty}.
\end{proof}

\subsubsection{Necessity}
\begin{proof}[Proof of the necessity of Theorem \ref{thm:disjoint}]
We need to prove that if $\JJJ_1\cap\JJJ_2=\emptyset$, then $\g$ is a coiling curve. Recall that $\AAA_{\PPP}$ is the multi-annulus isotopic to $\G_{\PPP}$ rel $\PPP$ and $A(\g)$ is the annulus in $\AAA_{\PPP}$ isotopic to $\g$ rel $\PPP$. As we have constructed in \S \ref{sec:sufficiency}, $W_1$ and $W_2$ are the two components of $\cbar\sm\AAA_{\PPP}$ attached at $A(\g)$ and $F$ is the marked branched covering. Then $F$ is Thurston equivalent to $f^m$ rel $\PPP$ and hence has no Thurston obstructions, where $m\geq 1$ is the integer such that each periodic component of $\cbar\sm\G_{\PPP}$ is fixed under $f^m_{\G_{\PPP}}$. Moreover essential components of $F^{-1}(\AAA_{\PPP})$ is contained in $\AAA_{\PPP}$ and $\pa\AAA_{\PPP}\subset\pa F^{-1}(\AAA_{\PPP})$. 

Assume that $\g$ is not coiling in $\G$. Then $\kappa_n(\g,\G_{\PPP})=1$ for all $n\geq 1$ by the construction of $\G_{\PPP}$. Thus there is only one essential component of $F^{-1}(\AAA_{\PPP})$ isotopic to $\g$ rel $\PPP_f$. Since $A(\g)\cap\PPP_f=\emptyset$ and $F(\pa\AAA_{\PPP})\subset\pa\AAA_{\PPP}$, so $A(\g)$ has no intersection with nonessential components of $F^{-1}(\AAA_{\PPP})$. 

Suppose $\g$ is periodic. Then its period is one by the choice of $m$. Then by the above argument $F(A(\g))=A(\g)$. Moreover the degree $\deg F|_{A(\g)}$ is at least $2$. For otherwise, the cycle of $\{\g\}$ is a Levy cycle of $F$, which contradicts that $F$ has no Thurston obstructions. 


Recall that $h$ is the semiconjugacy from $F$ to $f^m$ and $h^{-1}(\JJJ_i)$ is compactly contained in $W_i$, where $i=1,2$. Denote by $A$ the annulus bounded by $h^{-1}(\JJJ_1)$ and $h^{-1}(\JJJ_2)$. By the construction of $F$, we know that $F(A)=A$ and $\deg F|_{A}=\deg F|_{A(\g)}\geq 2$. 

If $\JJJ_1\cap\JJJ_2=\emptyset$, then $h(A)$ is an annulus bounded by $\JJJ_1$ and $\JJJ_2$ and has positive modulus. Moreover, $f^m(h(A))=h(A)$ and $\deg f^m|_{h(A)}=\deg F|_{A}\geq 2$. These is a contradiction since $h(A)$ has positive modulus. Hence $\JJJ_{1}\cap\JJJ_{2}\neq\emptyset$. 

Now suppose that $\g$ is not periodic. Since $\kappa_n(\g,\G_{\PPP})=1$ for all $n\geq 1$, there is a unique periodic curve $\alpha$, which is not coiling, such that $A(\g)$ is a component of $F^{-n}(A(\alpha))$ for some $n\geq 1$. Assume that $h(A(\g))$ is an annulus with positive modulus. Then $h(A(\alpha))=f^{nm}(h(A(\g)))$ is an annulus with positive modulus. This contradicts the conclusion in the periodic case. So in the case that $\g$ is not periodic, $\JJJ_{1}\cap\JJJ_{2}\neq\emptyset$ is also true. This ends the proof of the necessity of Theorem \ref{thm:disjoint}. 
\end{proof}


\section{Renormalization}
\subsection{Rational-like maps}\label{sec:rational-like}
At first, we recall some concepts and results about the renormalizations. A proper map $h:U\ra V$ of degree $d\geq 2$ is a polynomial-like map if $\,U,V\subset\mathbb{C}$ are two simply connected domains and $U$ is compactly contained in $V$. The filled Julia set of $h$ is denoted by 
$$
K:=\bigcap\limits_{n=1}^{\infty}h^{-n}(V). 
$$
By the straightening theorem (see \cite[Theorem 1]{douady1985dynamics} or \cite[Theorem 5.7]{mcmullen1994complex}), there is a polynomial $g$ of degree $d$ and a quasiconformal map $\phi$ of $\mathbb{C}$ such that 
\begin{itemize}
\item[(1)] $\JJJ_g=\pa\phi(K)$, 
\item[(2)] $g\circ\phi=\phi\circ h$ in a neighborhood of $K$, and
\item[(3)] the complex dilatation $\mu_{\phi}$ of $\phi$ satisfies $\mu_{\phi}(z)=0$ for a.e. $z\in K$. 
\end{itemize}
When $K$ is connected, the polynomial $g$ is unique up to affine conjugation. If $h=f^p|_U$ for some rational map $f$ and $\deg h<\deg f^p$, $g$ is called a renormalization of $f$. The boundary $\pa K$ is called a small Julia set. 

For rational maps with disconnected Julia sets, McMullen defined stable conjugate relation on their periodic Julia components \cite{mcmullen1988automorphisms}. He proved that on a periodic Julia component of a rational map $f$ with period $p\geq 1$, $f^p$ is stably conjugate to a rational map $g$ whose Julia set is connected. The map $g$ can be viewed as a ``renormalization" of $f$. In this paper, we adopt a stronger version which is defined by rational-like maps \cite{buff2003virtually, cui2016renormalizations}. 

A  \textbf{rational-like} map is a holomorphic proper map $h:U\ra V$ of degree $d\geq 2$, satisfying that 
\begin{itemize}
\item[(1)] $U,V$ are two finitely connected domains of $\cbar$ and $U$ is compactly contained in $V$, 
\item[(2)] the orbit of every critical point of $h$ stays in $U$, and
\item[(3)] each component of $\cbar\sm U$ contains at most one component of $\cbar\sm V$. 
\end{itemize}
The {\bf filled Julia set} of $h$ is defined as $$K_h:=\bigcap\limits_{n\geq 1}h^{-n}(V).$$
By \cite[Proposition 5.1]{cui2016renormalizations}, $K_h$ is connected. 
The following straightening theorem (see \cite[Theorem 4]{buff2003virtually} or \cite[Theorem 5.2]{cui2016renormalizations}) holds for rational-like maps. 

\begin{thm}[\bf Straightening]\label{thm:straightening}
Let $h:U\ra V$ be a rational-like map and $K_h$ be its filled Julia set. Then there is a rational map $g$ with $\deg g=\deg h$ and a quasiconformal map $\phi$ of $\cbar$ such that:
\begin{itemize}
\item[(1)] $\JJJ_g=\pa\phi(K_h)$, 
\item[(2)] $g\circ\phi=\phi\circ h$ in a neighborhood of $K_h$, and 
\item[(3)] the complex dilatation $\mu_{\phi}$ of $\phi$ satisfies $\mu_{\phi}(z)=0$ for a.e. $z\in K_h$. 
\end{itemize}
Moreover, $g$ is unique up to holomorphic conjugation. 
\end{thm}
We say $h$ is {\bf hybrid equivalent} to $g$. Suppose that $f$ is a rational map and there are domains $U, V$ and an integer $p\geq 1$ such that $f^p:U\ra V$ is a rational-like map. Then by Theorem \ref{thm:straightening}, there is a rational map $g$ such that $f^p:U\to V$ is hybrid equivalent to $g$. The unique rational map $g$ is called a \textbf{renormalization} of $f$ if $\deg g < \deg f^p$. The set $\JJJ_g$ is called a {\bf small Julia set}. 
\vskip 0.24cm

{\bf Remark:} The definition of rational-like maps in \cite{buff2003virtually} only requires the condition (1). The straightening theorem is still true in this case except that the rational map $g$ is not unique. In this paper, we use rational-like maps to define renormalizations of rational maps, which usually have connected filled Julia sets. So it is convenient to require the other two conditions. 

\subsection{Proof of Theorem \ref{thm:renorm}}
Before we prove Theorem \ref{thm:renorm}, we first give a sufficient condition for the existence of renormalizations of rational maps. 

\begin{lem}\label{lem:renorm}
Suppose that $f$ is a PCF rational map with nonempty Fatou set and $E$ is a non-degenerate continuum with $\pa E\subset \JJJ_f$ such that $f^p(E)=E$ for some integer $p\geq 1$. If each component of $f^{-p}(E)$ is either $E$ or disjoint from $E$, then there is a neighborhood $\,V$ of $E$ such that 
\begin{itemize}
\item[(1)] $f^p: V^p\to V$ is a rational-like map, where $V^p$ is the component of $f^{-p}(V)$ containing $E$;
\item[(2)] $E=\bigcap\limits_{n\geq1}f^{-np}(V)$. 
\end{itemize}
\end{lem}
\begin{proof}
At first we claim that $\deg f^p|_E \geq 2$. Otherwise $f^p: E \ra E$ would be a bijection. Following the same argument in the proof of \cite[Lemma 18.8]{milnor2006dynamics}, the expansion of $f^p$ in terms of the orbifold metric on $\pa E\subset\mc{J}_f$ implies that $\pa E$ is a finite set, which contradicts the fact that $E$ is a non-degenerate continuum. To simplify the notations, we may assume that $p=1$ in the following. 

Let $\wt{E}$ be the union of $E$ and the components of $\,\cbar \sm E$ which are disjoint with $\PPP_f$ and $\DD$ be the collection of components of $\,\cbar\sm\wt{E}$. Then $\DD$ is a non-empty collection since the Fatou set of $f$ is nonempty. We define a self-map $f_{\#}$ on $\DD$ as follows. If $D\in\DD$ is disjoint from $f^{-1}(E)$,  then $f(D)\in\DD$ and we set $f_{\#}(D):=f(D)$. Otherwise, let $D'$ be the component of $D\sm f^{-1}(E)$ such that $\partial D'\supset\partial D$. In this case $f(D')$ is an element of $\DD$, and we define $f_{\#}(D):=f(D')$.

Since $\DD$ is a finite collection, each of its element is eventually periodic under $f_{\#}$. Assume that $D_i$ ($0\le i<m$) is a cycle in $\DD$ with $D_i=f_{\#}^i(D_0)$ and $D_0=f_{\#}^m(D_0)$. Since $f$ is expanding in a neighborhood of $\JJJ_f$ under the orbifold metric and $\partial E\subset \JJJ_f$, for each $1\le i\le m$, there is an annulus $A_i\subset D_i\sm \PPP_f$ with $\partial D_i\subset\partial A_i$, such that $\ol{A_i^1}\subset A_i\cup\partial D_i$, where $A_i^1$ is the component of $f^{-1}(A_{i+1})$ (set $A_m=A_0$) with $\partial A_i^1\supset\partial D_i$. With a similar argument, we may assign an annulus $A_{D}$ for every periodic element $D$ of $\DD$.

 If $D'\in\DD$ is not $f_{\#}$-periodic but $f_{\#}(D')=D$ is periodic, we assign an annulus $A_{D'}\subset D'\sm \PPP_f$ with $\partial D'\subset\partial A_{D'}$, such that $\ol{A_{D}^1}\subset A_{D'}\cup\partial D'$, where $A_{D}^1$ is the component of $f^{-1}(A_{D})$ with $\partial D'\subset A_{D}^1$. Repeat this assignment, we assign an annulus $A_{D}$ for each element $D\in\DD$. Moreover $A_D$ can be chosen with the property that $A_D\cap f^{-1}(E)=\emptyset$ since any component of $f^{-1}(E)\sm E$ has no intersection with $E$. 

Let $V$ be the union of $\wt{E}$ and $A_{D}$ for all $D\in\DD$. Then $V$ is a finitely connected domain. Moreover, the component $V^1$ of $f^{-1}(V)$ which contains $E$ is compactly contained in $V$ by the construction of $A_{D}$. It is clear that the orbit of every critical points of $f$ in $V^1$ stays in $V^1$ and each component of $\cbar\sm V^1$ contains at most one component of $\cbar\sm V$. Therefore $f:V^1\to V$ is a rational-like map. 

To simplify the notations, we denote $$L:=\bigcap\limits_{n\geq1}f^{-n}(V).$$ Now we want to prove that $E=L$. It is obvious that $E\subset L$. To obtain $L\subset E$, we only need to prove that $L\sm E=\emptyset$.  Assume $L\sm E\neq\emptyset$. Let $Y$ be a component of $L\sm E$. Because $\pa E\subset\JJJ_f$ and $\pa L\subset\JJJ_f$, so as long as $Y$ contains a point $z\in\FFF_f$, it would contain the Fatou component $U$ which contains $z$. Since $(f^{-1}(E)\sm E)\cap V=\emptyset$ and $U\subset L\subset V$, we have $f^k(U)\subset L\sm E$ for all $k\geq 0$. Note that there exists an integer $k\geq 0$ such that $f^k(U)$ is periodic and hence $f^k(U)\cap\PPP_f\neq \emptyset$. On the other hand, $(L\sm E)\cap \PPP_f\subset (V\sm E)\cap \PPP_f=\emptyset$. This is a contradiction and hence $Y\subset\JJJ_f$. 

If $Y$ is not an empty set, there exists $n\geq 1$ such that $f^n(Y)\cap E\neq\emptyset$ and $f^{n-1}(Y)\cap E=\emptyset$ \cite[Theorem 4.10]{milnor2006dynamics}. Hence $f^{n-1}(Y)\subset f^{-1}(E)\sm E$. Since any component of $f^{-1}(E)\sm E$ has positive distance from $E$,  we conclude that $f^{n-1}(\ol{Y})\cap E=\emptyset$. Note that $f(E)=E$ and $\ol{Y}\cap E\neq \emptyset$. Hence $f^{n-1}(\ol{Y})\cap E\neq\emptyset$. This is a contradiction. So $Y$ is an emptyset, which implies that $L\sm E=\emptyset$. The proof is complete.   
\end{proof}

\begin{proof}[Proof of Theorem \ref{thm:renorm}]
Since $U\in\UU_{\G}$ is periodic with period $p\geq 1$, there is a rational map $g$ which is the combinatorial renormalization of $f$ on $U$ by Theorem \ref{thm:combinatorial}. As in the proof of Theorem \ref{thm:combinatorial}, let $\AAA$ be the multi-annulus isotopic to $\G$ rel $\PPP_f$, $F$ be the standard form of $f$ with respect to $\G$ and $W$ be the component of $\cbar\sm\AAA$ isotopic to $U$ rel $\PPP_f$. Then $W^p$, the component of $F^{-p}(W)$ parallel to $W$, is contained in $W$ and $\pa W\subset\pa W^p$. Denote by $W^{np}$ the component of $F^{-np}(W)$ with the same property and 
$$K_U=\bigcap_{n\geq 1}\ol{W^{np}}. $$

Let $h$ be the semi-conjugacy from $F$ to $f$. Now we are going to prove Theorem \ref{thm:renorm} by checking that $h(K_U)$ satisfies the conditions in Lemma \ref{lem:renorm}. Since $F^p(K_U)=K_U$ and $h\circ F=f\circ h$, we have $f^p(h(K_U))=h(K_U)$. 

Denote by $\G_U\subset \G$ the completely stable multicurve generated by curves in $\pa U$. Then by definition each curve in $\G_U$ is coiling in $\G$ since each component of $\pa U$ is coiling in $\G$. Denote by $\AAA_U$ the collection of annuli in $\AAA$ isotopic to $\G_U$. Since $\G_U$ is a completely stable multicurve and each curve in it is coiling in $\G$, Lemma \ref{lem:key} is still true if we replace $\AAA_0$ by $\AAA_U$ in its statement. Hence for any two distinct components $E_1$ and $E_2$ of $\cbar\sm F^{-np}(\AAA_U)$, $h(E_1)\cap h(E_2)=\emptyset$. On the other hand, different components of $F^{-np}(K_U)$ are contained in different components of $\cbar\sm F^{-np}(\AAA_U)$. Hence for any components $K$ of $f^{-np}(h(K_U))$, either $K\cap h(K_U)=\emptyset$ or $K=h(K_U)$. So $h(K_U)$ satisfies the conditions in Lemma \ref{lem:renorm}. Hence there is a neighborhood $V$ of $h(K_U)$ and a component $V^p$ of $f^{-p}(V)$ such that $f^p: V^p\to V$ is a renormalization. Moreover, the straightening map of $f^p:V^p\to V$ is holomorphically conjugate to the combinatorial renormalization $g$ of $f$ on $U$ since they are both Thurston equivalent to the branched covering $G$ which is constructed in the proof of Theorem \ref{thm:combinatorial}. The proof of Theorem \ref{thm:renorm} is complete. 
\end{proof}

\subsection{Proof of Theorem \ref{thm:submulticurve}}\label{sec:domain}
Let $f$ be a PCF rational map and $\G$ be a completely stable multicurve containing coiling curves. As we have seen in \S \ref{sec:combinatorial}, $\G$ gives a combinatorial decomposition of $f$ and induces a dynamics $f_{\G}$ on the collection $\UU_{\G}$ of components of $\cbar\sm\G$. Theorem \ref{thm:renorm} shows that if there is a periodic component $U\in\UU_{\G}$ such that each component of $\pa U$ is a coiling curve of $\G$, then $U$ induces a renormalization of $f$. However, there are examples such that any periodic component in $\UU_{\G}$ is not bounded by coiling curves of $\G$ (see Figure \ref{fig:falsegeneral}). On the other hand, if $\G$ contains a Cantor sub-multicurve $\G'$, then $\G'$ is a completely stable multicurve such that each curve of $\G'$ is coiling in $\G'$. Hence every periodic piece of $f_{\G'}$ is bounded by coiling curves of $\G'$ and Theorem \ref{thm:submulticurve} is a direct corollary of Theorem \ref{thm:renorm}. In this section, we deal the problem under the following assumption. 
\vskip 0.24cm 

\noindent\centerline {\bf Assumption: $\G$ contains no Cantor multicurves. }

\begin{figure}[htbp]
\centering
\includegraphics[width=8cm]{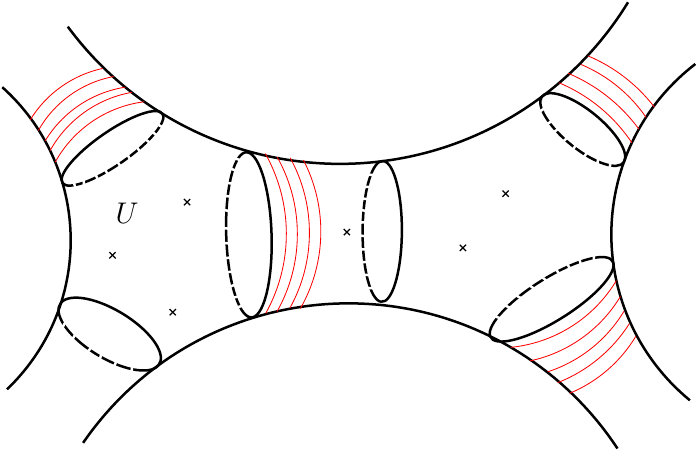}
\caption{Every piece in this picture is not bounded by coiling curves (marked black with some red isotopic curves nearby) in $\G$.}
\label{fig:falsegeneral}
\end{figure}

\vskip 0.24cm
In this case we cannot find a completely stable submulticurve $\G'$ such that each curve in $\G'$ is coiling in $\G'$. 
Fortunately, we can prove that there exists a submulticurve $\G'$ of $\G$ such that some periodic component $U\in\UU_{\G'}$ is bounded by coiling curves of $\G'$. Recall that a cycle containing a periodic curve is defined in \S \ref{sec:dynamicsoncurves}. 

\begin{lem}\label{lem:uniquecycle}
If $\G$ contains no Cantor multicurves, then for any periodic curve $\g\in\G$, the cycle containing $\g$ is unique. 
\end{lem}
\begin{proof}
If there are two cycles $\OO_1=\{\g_0=\g,\g_1,\cdots,\g_{p_1-1}\}$ and $\OO_2=\{\alpha_0=\g,\alpha_1,\cdots,\alpha_{p_2-1}\}$ containing $\g$, then $F^{-p_1p_2}(\g)$ has at least two components isotopic to $\g$. By definition, $\kappa_{kp_1p_2}(\g,\OO_1\cup\OO_2)\geq 2^k$. So $\kappa_{k\cdot p_1\cdot p_2}(\g,\OO_1\cup\OO_2)\to\infty$ as $k\to\infty$. Thus the completely stable multicurve generated by $\OO_1\cup\OO_2$ is a Cantor multicurve. This is a contradiction. 
\end{proof}
When $\G$ contains no Cantor multicurves, the number of curves in $\OO_{\g}$ is called the {\bf period} of $\g$. Recall that for each periodic curve $\g\in\G$, $\Lambda_{\g}$ is the completely stable multicurve generated by $\g$ as follows: $\beta\in\Lambda_{\g}$ if there exists an integer $k\geq 1$ such that $F^{-k}(\g)$ contains a curve isotopic to $\beta$ rel $\PPP$. Since $\G$ contains no Cantor multicurves, $\Lambda_{\g}$ contains no Cantor multicurves. The above argument shows that for any periodic curve $\alpha\in\Lambda_{\g}$ and any integer $n\geq 1$, $f^{-n}(\OO_{\alpha})$ has one and only one component, denoted by $\alpha^n$, isotopic to $\alpha$ rel $\PPP_f$. 

\begin{proof}[Proof of Theorem \ref{thm:submulticurve}]
Let $\G$ be a completely stable multicurve of $f$ containing no Cantor multicurves but containing a coiling curve. By Lemma \ref{lem:aperiodic-coiling}, $\G$ contains periodic coiling curves. Since $\G$ contains no Cantor multicurves, the period of each periodic curve is well-defined by Lemma \ref{lem:uniquecycle}. Suppose that $\g\in\G$ is a periodic coiling curve with period $p\geq 1$. To simplify the notations, we deform $f$ to a marked branched covering $F$ as follows. 

Let $\theta$ be a homeomorphism of $\cbar$ such that it is isotopic to the identity rel $\PPP_f$ and maps $\alpha$ to $\alpha^1$ for all periodic curve $\alpha\in\Lambda_{\g}$. Set $F=f\circ\theta$. Then $\PPP_F=\PPP_f$ and $F$ is Thurston equivalent to $f$. Moreover $F^p(\g)=\g$. Choose an orientation for $\g$. If $F^p$ reverses the orientation on $\g$, then $F^{2p}$ preserves the orientation on $\g$. Note that any completely stable multicurve under $F$ remains completely stable under $F^2$. Thus we only need to consider the case that $F^p$ is orientation-preserving on $\g$. To simplify the notation, we assume that $p=1$ in the following discussion. 

\begin{lem}\label{lem:oneside}
If $F: \g\to\g$ is orientation-preserving, then for each curve $\alpha\in\G\sm\{\g\}$, all curves in $\{F^{-n}(\alpha)\}_{n\geq 1}$ isotopic to $\g$ lies in the same component of $\cbar\sm\g$. 
\end{lem}
\begin{proof}
We first prove that if the conclusion is true for all periodic curves, then it is true for all curves in $\G\sm\{\g\}$. In fact, if $\alpha$ is not periodic, then by Lemma \ref{lem:pre-periodic} there is a period curve $\beta\in\G$, such that some components of $F^{-k}(\beta)$ are isotopic to $\alpha$ rel $\PPP_f$ for some integer $k\geq 1$ since $\G$ is pre-stable. So if some components of $F^{-n}(\alpha)$ are isotopic to $\g$ rel $\PPP_f$, then some components of $F^{-n-k}(\beta)$ are isotopic to $\g$ rel $\PPP_f$. Denote by $A$ the smallest annulus, may be degenerate, containing $\alpha$ and the components of $F^{-k}(\beta)$ which are isotopic to $\alpha$. Then $A\cap\PPP_f=\emptyset$. If by assumption $F^{-n-k}(\beta)$ lies in the same component of $\cbar\sm\G$ for all $n\geq 1$, then the components of $F^{-n}(A)$ isotopic to $\g$ lies in the same component of $\cbar\sm\g$. Otherwise, there is a component $A^n$ of $F^{-n}(A)$ containing $\g$ and then $F^{n}(\g)\subset A$. On the other hand, $F^n(\g)=\g$ and $\g$ is not isotopic to $\alpha$ rel $\PPP_f$. This is a contradiction. This implies that the conclusion is true for $\alpha$. 

\begin{figure}[htbp]
\centering
\includegraphics[width=5cm]{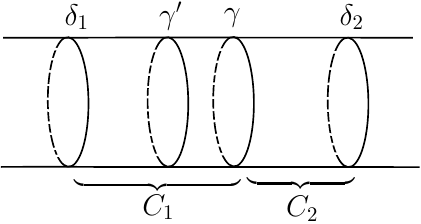}
\caption{$\g'$ is another component of $F^{-n}(\g)$.}
\label{fig:Lambda}
\end{figure}

Now we assume that $\alpha$ is periodic and some component of $F^{-n}(\alpha)$ is isotopic to $\g$ rel $\PPP_f$. Then the number of curves in $F^{-n}(\alpha)$ tends to infinity as $n$ tends to infinity by Lemma \ref{lem:periodcoiling}. If the curves in $F^{-n}(\alpha)$ which are isotopic to $\g$ rel $\PPP_f$ are separated by $\g$, we can choose two components $\delta_1$ and $\delta_2$ of them such that the annulus $A$ bounded by $\delta_1 $ and $\delta_2$ is disjoint from any essential components of $F^{-n}(\alpha)$ and $\g\subset A$. Denote by $B$ the annulus bounded by $\g$ and $\alpha$. Then $F^{-n}(B)$ has two components $B_1$ and $B_2$ which have disjoint closure such that $B_1$ is contained in the annuli $C_1$, which is bounded by $\delta_1$ and $\g$, and $B_2$ is contained in $C_2$, which is bounded by $\g$ and $\delta_2$. Since there are no essential components of $F^{-n}(\alpha)$ contained in $A$, it follows that both $\pa B_1$ and $\pa B_2$ contain components of $F^{-n}(\g)$ isotopic to $\g$ rel $\PPP_f$ and they are disjoint with each other. This implies that $F^{-n}(\g)$ has at least two components isotopic to $\g$ rel $\PPP_f$. This contradicts the fact that $\Lambda_{\g}$ is not a Cantor multicurve. So for any fixed $n\geq 1$, $F^{-n}(\alpha)$ which is isotopic to $\g$ rel $\PPP_f$ lies in one and the same component $D$ of $\cbar\sm\g$. 

Since $F(\g)=\g$ and $F$ is orientation-preserving on $\g$, some curves in $F^{-n-1}(\alpha)$ are isotopic to $\g$ rel $\PPP_f$ and lie in $D$. By the above discussion, all curves in $F^{-n-1}(\alpha)$ isotopic to $\g$ rel $\PPP_f$ lies in $D$. So for all $n\geq 1$, all curves in $\{F^{-n}(\alpha)\}_{n\geq 1}$ isotopic to $\g$ lies in one and the same component of $\cbar\sm\g$. 
\end{proof}
We may assume that there is a periodic curve $\alpha\in\G$ such that the components of $F^{-n}(\alpha)$ isotopic to $\g$ rel $\PPP_f$ lie in the left side of $\g$. For each curve $\beta\in\G$, denote by $\CC_{n}(\g,\beta)$ the curves in $F^{-n}(\beta)$ isotopic to $\g$ rel $\PPP_F$. Set 
$$\G^*=\{\beta\in\G: \CC_n(\g,\beta) \text{ lie in the left side of }\g\}. $$
Then $\G^*\neq\emptyset$ is completely stable and each curve in $\Lambda_{\g}$ is coiling in $\G^*$. Recall that $\CC_n(\g,\G^*)$ is the collection of curves in $F^{-n}(\G^*)$ which is isotopic to $\g$ rel $\PPP_f$. Then all curves in $\CC_n(\g,\G^*)\sm \{\g\}$ lie in the same component of $\cbar\sm \g$. 

Since $\kappa_n(\g,\Lambda_{\g})=1$ for all $n\geq 1$, both of the components of $\cbar\sm\Lambda_{\g}$ having $\g$ as a boundary component is fixed under the map $f_{\Lambda_{\g}}$. Let $U_{\g}$ be the one lying in the right side of $\g$. Then it is disjoint from $\CC_n(\g,\G^*)$. Denote by $\UUU_{\g}$ the union of all the components of $\cbar\sm\Lambda_{\g}$ in the grand orbit of $U_{\g}$ under $f_{\Lambda_{\g}}$ and 
$$
\G'=\{\beta\in\G^*:\, \beta\text{ is disjoint from }\UUU_{\g}\}.
$$
Because $\Lambda_{\g}$ is exactly the collection of the boundary components of $\UUU_{\g}$ and $\UUU_{\g}$ is an open set, any curve in $\Lambda_{\g}$ is disjoint from $\UUU_{\g}$. Hence $\Lambda_{\g}\subset\G'$. 
\begin{lem}\label{lem:one-side}
The following statements hold.
\begin{itemize}
\item[(1)] $\G'$ is completely stable.
\item[(2)] $U_{\g}\in\UU_{\G'}$ is fixed by $f_{\G'}$.
\item[(3)] Each curve in $\Lambda_{\g}$ is a coiling curve of \,$\G'$.
\end{itemize}
\end{lem}

\begin{proof}
(1) For each $n\ge 0$, let $\UUU_{\g}^n$ be the union of all the complex type components of $\cbar\sm f^{-n}(\Lambda_{\g})$ parallel to some component of $\UUU_{\g}$ rel $\PPP_f$. Then each component of $f(\UUU_{\g}^{n+1})$ is a component of $\UUU^n_{\g}$. Conversely, each complex type component of $f^{-1}(\UUU^n_{\g})$ is a component of $\UUU^{n+1}_{\g}$.

For any $\beta\in\G'$, $\beta$ is disjoint from $\UUU_{\g}$. Thus $f^{-1}(\beta)$ is disjoint from $\UUU_{\g}^1$. Therefore each essential curve in $f^{-1}(\beta)$ is isotopic to a curve in $\G'$ rel $\PPP_f$. So $\G'$ is a stable multicurve.

Conversely, since $\G^*$ is pre-stable, there is a curve $\alpha\in\G^*$ such that $f^{-1}(\alpha)$ contains a curve $\beta^1$ isotopic to $\beta\in\G'$ rel $\PPP_f$. Either $\alpha\subset\UUU_{\g}$ or $\alpha$ is disjoint from $\UUU_{\g}$. In the latter case, $\alpha\in\G'$. In the former case, $\beta^1\subset\UUU_{\g}^1$. So $\beta^1$ cannot be isotopic to any curve in $\G'\sm\Lambda_{\g}$ rel $\PPP_f$ and hence $\beta\in\Lambda_{\g}$. Noticing that $\Lambda_{\g}$ is pre-stable, there is a curve $\alpha'\in\Lambda_{\g}\subset\G'$ such that $f^{-1}(\alpha')$ contains a curve isotopic to $\beta$ rel $\PPP_f$. Therefore $\G'$ is pre-stable.

(2) By the definition of $\G'$, the component $U_{\g}$ of $\cbar\sm\Lambda_\g$ is also a component of $\cbar\sm\G'$. Recall that $U_{\g}$ is a fixed component of $f_{\Lambda_{\g}}$, i.e., $f(U_{\g}^1)=U_{\g}$, where $U_{\g}^{1}$ is the component of $\cbar\sm f^{-1}(\Lambda_{\g})$ parallel to $U_{\g}$ rel $\PPP_f$. As above, for any $\beta\in\G'$, $f^{-1}(\beta)$ is disjoint from $\UUU_{\g}^1$. Thus $U^1_{\g}$ is a component of $\cbar\sm f^{-1}(\G')$. So $U_{\g}\in\UU_{\G'}$ is fixed by $f_{\G'}$.

\begin{figure}[htbp]
\centering
\includegraphics[width=11cm]{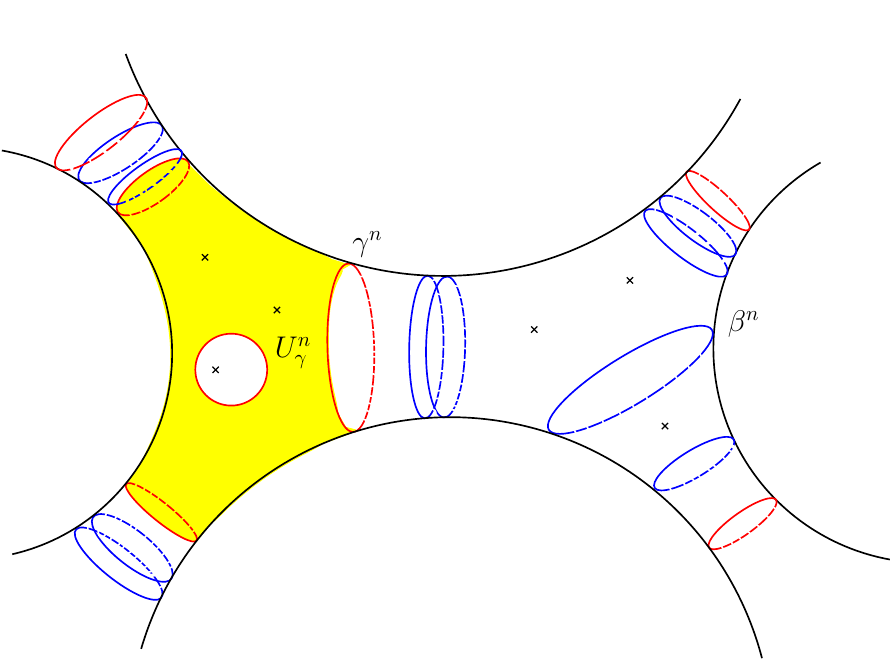}
\caption{The curve $\g$ is a coiling curve of $\G'$, which is isotopic to $\g^n$ rel $\PPP_f$. The curve $\beta$ is a curve in $\G'\sm\Lambda_{\g}$, which is isotopic to $\beta^n$ and makes $\g$ coiling in $\G'$. All of the red curves are components of $f^{-n}(\g)$ and all blue curves are components of $f^{-n}(\beta)$ for some integer $n\geq 1$. }
\label{fig:Lambda}
\end{figure}

(3) Note that each curve in $\Lambda_{\g}$ is isotopic to a curve in $f^{-k}(\g)$ rel $\PPP_f$ for some integer $k\ge 1$ and for any essential curve $\alpha\in f^{-k}(\g)$, $\kappa_{n+k}(\alpha,\G)\geq\kappa_n(\g,\G)$. Thus we only need to show that $\g$ is a coiling curve of $\G'$.

For any curve $\alpha\in\G^*\cap\UUU_{\g}$, each essential curve $\delta$ in $f^{-n}(\alpha)$ is contained either in $\UUU_{\g}^{n}$ or in an annular type component $A$ of $\cbar\sm f^{-n}(\Lambda_{\g})$. In the former case $\delta$ is not isotopic to $\g$ rel $\PPP_f$ since $\UUU^n_{\g}$ is disjoint from $\CC_n(\g,\G^*)$. In the latter case, let $\delta_0$ and $\delta_1$ be the two components of $\partial A$. Then both of them are isotopic to the essential curve $\delta$ which is contained in $A$ rel $\PPP_f$ and contained in $f^{-n}(\Lambda_{\g})$. If $\delta$ is isotopic to $\g$ rel $\PPP_f$, then $f^{-n}(\Lambda_{\g})$ would contain two distinct curves isotopic to $\g$ rel $\PPP_f$. This is impossible by Lemma \ref{lem:unique}. In summary, for each curve $\alpha\in\G^*\sm\G'$, $f^{-n}(\alpha)$ contains no curves isotopic to $\g$ rel $\PPP_f$. So $\kappa_n(\g,\G')=\kappa_n(\g,\G^*)\to\infty$ as $n\to\infty$. In other words, $\g$ is a coiling curve of $\G'$. 
\end{proof}
Now we get a completely stable multicurve $\G'$ and a periodic component $U$, which is the component $U_{\g}$ in the Lemma \ref{lem:one-side}, such that each component of $\pa U$ is a coiling curve in $\G'$. Hence by Theorem \ref{thm:renorm}, $U$ induces a renormalization of $f$. This completes the proof of Theorem \ref{thm:submulticurve}. 
\end{proof}

\section{Rational Maps with coiling curves}\label{sec:example}
In this section, by tuning and disk-annulus surgery, we will construct rational maps with completely stable multicurves which contain coiling curves but no Cantor multicurves. A byproduct is a sufficient condition for a Fatou domain to be a Jordan domain (refer to Theorem \ref{thm:coiledJordan}). 

\subsection{Tuning}
Tuning operation between polynomials was first introduced by Douady and Hubbard \cite{douady1985dynamics}. Pilgrim studied the tuning of PCF rational maps with PCF polynomials in his thesis \cite{pilgrim1995cylinders}. We will introduce the tuning on a fixed Fatou domain of PCF rational maps for simplicity. 

Let $R$ be a PCF rational map with a fixed Fatou domain $D$ and $\phi$ be a B\"{o}ttcher coordinate of $D$, i.e., $\phi$ is a conformal map from $D$ onto the unit disk $\mathbb{D}$ such that $\phi\circ R\circ\phi^{-1}(z)=z^d$, where $d:=\deg R|_D$. Then $\phi^{-1}$ can be extended continuously to $\overline{\mathbb{D}}$ since $\partial D$ is locally connected \cite[Theorem 19.7]{milnor2006dynamics}.

Let $P$ be a monic PCF polynomial with degree $\deg P=d$ such that $\PPP_P$ contains the origin. Let $\chi(z)=|z|z/(1-|z|)$. Then $\chi:\,\mathbb{D}\to\mathbb{C}$ is a homeomorphism. Thus $\chi^{-1}\circ P\circ\chi$ is a branched covering of $\mathbb{D}$, which can be continuously extended to $\overline{\mathbb{D}}$ with boundary value $z\mapsto z^d$. Thus $\phi^{-1}\circ\chi^{-1}\circ P\circ\chi\circ\phi$ is a branched covering of $D$ whose boundary value coincides with $R$.

The {\bf formal tuning} of $R$ with $P$ in $D$ is defined by
$$
G=
\begin{cases}
R & \text{on }\cbar\sm D, \\
\phi^{-1}\circ\chi^{-1}\circ P\circ\chi\circ\phi & \text{on }D.
\end{cases}
$$
Since $\PPP_P$ contains the origin, we have $\PPP_G=(\phi^{-1}\circ\chi^{-1}(\PPP_P))\cup\PPP_R$. Thus $G$ is a PCF branched covering. If $G$ has no Thurston obstructions, then it is Thurston equivalent to a rational map $g$, which is called the {\bf tuning} of $R$ with $P$ in $D$. When $\#\PPP_P=2$, the tuning of $R$ with $P$ is trivial and is Thurston equivalent to a $R$ itself. 

Here we remark that the B\"{o}ttcher coordinate $\phi$ is unique up to rotations with angles $2k\pi/(d-1)$ ($1\le k<d-1$). Thus the definition of formal tuning has $d-1$ choices when $d>2$.

\subsection{Coiled Fatou domains}\label{sec:coiled}
Let $(f,\PPP)$ be a marked rational map. 
A {\bf pseudo-multicurve} $\G\subset\cbar\sm \PPP$ is a finite collection of essential or peripheral curves that are pairwise disjoint and pairwise non-isotopic \cite{dudko2022canonical}. We would like to mention that for any pseudo-multicurve $\G$, there might exist components of $\cbar\sm\G$ that are not complex type. Therefore the definition of $\UU^1$ in \S \ref{sec:combinatorial} needs to be replaced by the collection of components of $\cbar\sm f^{-1}(\G)$ which parallel to some components of $\cbar\sm\G$ rel $\PPP_f$. By the same argument as before, we can prove the following result: 
\begin{thm}\label{thm:prenorm}
Let $f$ be a PCF rational map and $\G$ be a completely stable pseudo-multicurve. Suppose $U$ is a component of $\,\cbar\sm\G$ which is periodic with period $p\geq 1$ under $f_{\G}$. If every curve in $\pa U$ is a coiling curve in $\G$, then there exists a finitely connected domain $V$, isotopic to $U$ rel $\PPP_f$, and a component $V^p$ of $f^{-p}(V)$ such that $f^p:V^p\to V$ is a renormalization. 
\end{thm}

Based on this theorem, we can prove Theorem \ref{thm:coiledJordan}. 
\begin{proof}[Proof of Theorem \ref{thm:coiledJordan}]
We use the notations $\alpha,\beta,a$ and $D$ in Definition \ref{def:coilingdomain} before Theorem \ref{thm:coiledJordan}. We may assume that $\beta\subset D$. Denote by $\Lambda$ the pseudo-multicurve generated by $\alpha$. Then $\beta\in\Lambda$ and is coiling in $\Lambda$. Moreover the component of $\cbar\sm\beta$ which is disjoint from $\alpha$ has no intersection with $f^{-n}(\Lambda)$ for all $n\geq 1$. Taking $U$ in Theorem \ref{thm:prenorm} as the component of $\cbar\sm\Lambda$ containing the point $a$, then there is a simply connected domain $V$ whose boundary is isotopic rel $\PPP_f$ to $\beta$ and a component $V^1$ of $f^{-1}(V)$ such that $f:V^1\to V$ is a renormalization. Moreover, $V\cap\PPP_f=V^1\cap\PPP_f=\{a\}$. Hence $V^1$ is simply connected and $f:V^1\to V$ is a polynomial-like map whose post-critical set is a singleton. So $f:V^1\to V$ is hybrid equivalent to the map $g(z)=z^k$, where $k=\deg_af$. 

Denote $$\mc{K}=\bigcap_{n\geq 1}f^{-n}(V). $$
Then $\phi(\mc{K})=\ol{\mathbb{D}}$, where $\phi$ is the straightening map from $f: V^1\to V$ to $g(z)=z^k$. Hence $\mc{K}$ is a topologically disc with connected interior since $\phi$ is quasiconformal on $V^1$. Hence the interior of $\mc{K}$ is a Fatou domain containing $a$. So $D=\mc{K}\sm\pa\mc{K}$ is a Jordan domain. 

Now we are going to prove that for any other Fatou domain $\Omega$ of $f$, $\ol{D}\cap\ol{\Omega}=\emptyset$. Since $\pa D\cap\PPP_f=\emptyset$, the closure of $D$ is disjoint from the closure of any components of $f^{-n}(D)$ for all $n\geq 1$. Note that there is an integer $n\geq 0$ such that $f^n(\Omega)$ is periodic. If a Fatou domain $\Omega$ does not lie in any backward orbit of $D$ and $\ol{D}\cap\ol{\Omega}\neq\emptyset$, then $\ol{D}\cap\ol{f^n(\Omega)}\supset f^n(\ol{D}\cap\ol{\Omega})\neq\emptyset$. This implies $\ol{D}$ intersects the closure of some periodic Fatou domain. Hence we only need to check that $\ol{D}\cap\ol{\Omega}=\emptyset$ for each periodic Fatou domain $\Omega$ which is not $D$. 

As in \S \ref{sec:combinatorial}, denote by $\AAA$ the multi-annulus isotopic to $\Lambda$. Let $F$ be the standard form of $f$ corresponding to the multicurve $\Lambda$ and $h$ is the semi-conjugacy from $F$ to $f$. Suppose $A(\beta)\subset\AAA$ is the annulus isotopic to $\beta$ rel $\PPP_f$. Denote by $W_0$ the piece of $\cbar\sm\AAA$ which contains $a$ and $W_1$ the other piece of $\cbar\sm\AAA$ which intersects $\ol{A(\beta)}$. Then the annulus $A$ bounded by $h(W_0)$ and $h(W_1)$ has positive modulus since $\beta$ is a coiling curve by Lemma \ref{lem:key}. 
Let $\Omega$ be a periodic Fatou domain. Then the corresponding attracting periodic point $a_{\Omega}$ does not lie in $\ol{A}$ since $\ol{A}\cap\PPP_f=\emptyset$. So $h^{-1}(a_{\Omega})\notin\ol{\AAA}$. Because $F(\pa \AAA)\subset \pa \AAA$, thus $h^{-1}(\Omega)\cap\ol{\AAA}=\emptyset$. In other words, any periodic Fatou domain of $f$ has no intersection with $\ol{A}$. Therefore $A$ separates $\ol{D}$ from the closure of all other periodic Fatou domains. Hence the closure of a coiled Fatou domain is disjoint from the closure of any other Fatou domain. 
\end{proof}

\subsection{Special coiled Fatou domains}
In this section, we prove that a coiled Fatou domain coiled by a PCF polynomial (see Definition \ref{def:coilingdomain}) which is injective on its Hubbard tree can tune any PCF polynomial. 
\begin{thm}\label{thm:tunable}
Let $g$ be a PCF rational map with a fixed Fatou domain $D$ coiled by a PCF polynomial $Q$ and $d\geq 2$ be the degree of $g$ on $D$. Suppose $P$ is any monic PCF polynomial with degree $d$ and $\PPP_P$ contains the origin. If $Q$ is injective on its Hubbard tree, then the formal tuning $F$ of $g$ with $P$ in $D$ is Thurston equivalent to a rational map $f$.
\end{thm}

Refer to \cite[Definition 4.1]{orsay} for the definition of Hubbard trees.

\begin{proof}
By the definition of tuning, $\PPP_F=(\phi^{-1}\circ\chi^{-1}(\PPP_P))\cup\PPP_g$. Without loss of generality, we may assume that $\#(\PPP_P\sm\{\infty\})\geq 2$. Let $\beta\subset D\sm\PPP_F$ be a Jordan curve such that the annular component of $D\sm\beta$ is disjoint from $\PPP_F$. Then $\beta$ is an essential curve and $F^{-1}(\beta)$ contains a unique essential curve, which is isotopic to $\beta$ rel $\PPP_F$.

By the definition of coiled Fatou domain coiled by the polynomial $Q$, there is an essential curve $\alpha\subset\cbar\sm\PPP_F$ disjoint from $\beta$, such that $F^{-1}(\alpha)$ contains at least two essential curves, one of which is isotopic to $\alpha$ rel $\PPP_F$ and another is isotopic to $\beta$ rel $\PPP_F$. Denote by $U$ the component of $\cbar\sm\alpha$ disjoint from $\beta$.

Since $Q$ is injective on its Hubbard tree and is the combinatorial renormalization of $g$ on $U$, there is a tree $T\subset U$ containing $\PPP_F\cap U$ and a homeomorphism $\theta$ of $\cbar$ isotopic to the identity relative to a neighborhood of $\PPP_F$, such that $F$ is injective on $\theta(T)$ and $F(\theta(T))\subset T$.

By Lemma \ref{lem:irreducible}, we only need to prove that for any irreducible multicurve $\G$ of $F$, its leading eigenvalue $\lambda(M_{\G})<1$. For each curve $\g\in\cbar\sm\PPP_F$, denote
$$
k(\alpha, \g)=\min_{\g'\sim\g}\#(\alpha\cap\g'),\,\text{ and }\,k(\beta,\g)=\min_{\g'\sim\g}\#(\beta\cap\g'),
$$
where the minimum is taken over all the curves $\g'$ isotopic to $\g$ rel $\PPP_F$. There are following three cases: 

{\bf Case 1:} If there exists a curve $\g\in\G$ such that $k(\beta,\g)=0$, then for any curve $\delta$ in $F^{-1}(\g)$, $k(\beta,\delta)=0$ since $F^{-1}(\beta)$ contains a curve isotopic to $\beta$ rel $\PPP_F$.
By the irreducibility, $k(\beta,\g)=0$ for all $\g\in\G$. Thus the transition matrix of $\G$ under $F$ is equal to the transition matrix of $\G$ under $g$ or $Q$. Hence $\lambda(M_{\G})<1$.

\vskip 0.24cm
{\bf Case 2:} If there exists a curve $\g\in\G$ such that $k(\alpha,\g)=0$, then $k(\beta,\delta)=0$ for any curve $\delta$ contained in $F^{-1}(\g)$ since $F^{-1}(\alpha)$ contains a curve isotopic to $\beta$ rel $\PPP_F$. So we also have $\lambda(M_{\G})<1$ from the above discussion.
\vskip 0.24cm

{\bf Case 3:} If for each curve $\g\in\G$, both $k(\beta,\g)$ and $k(\alpha,\g)$ are non-zero. Denote
$$
k(T,\g)=\min_{\g'\sim\g}\#(T\cap\g'),
$$
where the minimum is taken over all the curves $\g'$ isotopic to $\g$ rel $\PPP_F$. Then $k(T,\g)\neq 0$ for all $\g\in\G$. For each $\g\in\G$, assume $\#(T\cap\g)=k(T,\g)$. Let $\delta_i$ ($1\le i\le n$) be all the curves in $F^{-1}(\g)$. Since $F$ is injective on $\theta(T)$ and $F(\theta(T))\subset T$, we obtain
$$
\sum_{i=1}^n k(T,\delta_i)\le\sum_{i=1}^n\#(\theta(T)\cap\delta_i)\le\#(T\cap\g)=k(T,\g).
$$
Combining this inequality with the irreducibility of $\G$, we obtain that $k(T,\g)$ is a constant for $\g\in\G$, and for any $\g\in\G$, $F^{-1}(\g)$ contains exactly one curve isotopic to a curve in $\G$ rel $\PPP_F$. If $\lambda(M_{\G})\ge 1$, then $\G$ contains a Levy cycle $\{\g_i\}$ ($1\le i\le p$) in $\G$, i.e., $F^{-1}(\g_{i+1})$ contains a curve $\g'_i$ isotopic to $\g_i$ rel $\PPP_F$ for $1\le i\le p$ (set $\g_{p+1}=\g_1$), and $\deg F|_{\g_i'}=1$.

Now assume $\#(\alpha\cap\g_{i+1})=k(\alpha,\g_{i+1})$. Since $\deg F|_{\g_i'}=1$, we have
$$
k(\beta,\g_i)+k(\alpha,\g_i)\le\#(F^{-1}(\alpha)\cap\g'_i)\le\#(\alpha\cap\g_{i+1})=k(\alpha,\g_{i+1}).
$$
This is impossible since $\{\g_i\}$ is a cycle and both $k(\beta,\g_i)$ and $k(\alpha,\g_i)$ are non-zero for all $1\le i\le p$. Thus $\lambda(M_{\G})<1$. So the formal tuning $F$ of $g$ with $P$ has no Thurston obstructions. By Theorem \ref{thm:thurston}, $F$ is Thurston equivalent to a rational map $f$.
\end{proof}

\subsection{The construction of special coiled Fatou domains}
Now we use disc-annulus surgery to construct a PCF rational map satisfying the conditions of Theorem \ref{thm:tunable}. We begin with a PCF rational map $R$ with a fixed Fatou domain $U$ such that $R^{-1}(U)$ has another component $V\neq U$. Let $Q$ be a monic PCF polynomial with $\deg Q=\deg R|_U$ such that $\#\PPP_Q\ge 3$, $\PPP_Q$ contains the origin and $Q$ is injective on its Hubbard tree $T$. Recall that the formal tuning of $R$ with $Q$ in $U$ is defined as
$$
G_0=
\begin{cases}
R & \text{on }\cbar\sm U, \\
\phi^{-1}\circ\chi^{-1}\circ Q\circ\chi\circ\phi & \text{on }U,
\end{cases}
$$
where $\phi$ is a B\"{o}ttcher coordinate of $U$ and $\chi(z)=|z|z/(1-|z|)$.

Pick a Jordan curve $\alpha\subset U$ separating $\phi^{-1}\circ\chi^{-1}(T)$ from $\partial U$. Then $R^{-1}(\alpha)$ contains a unique curve $\alpha_0$ in $U$, which also separats $\phi^{-1}\circ\chi^{-1}(T)$ from $\partial U$, and a unique curve $\beta_0$ in $V$, which separates the unique pre-fixed point $c\in V$ from $\partial V$. Pick a Jordan curve $\beta_1\subset V$ disjoint from $\beta_0$ and separating the point $c$ from $\beta_0$. Denote by
\begin{itemize}
\item $U(\alpha)$: the domain bounded by $\alpha$ in $U$,
\item $V(\beta_0)$: the domain bounded by $\beta_0$ in $V$,
\item $V(\beta_1)$: the domain bounded by $\beta_1$ in $V(\beta_0)$, and
\item $A(\beta_0,\beta_1)$: the annulus bounded by $\beta_0$ and $\beta_1$ in $V(\beta_0)$.
\end{itemize}
Define a branched covering $G:\,\cbar\to\cbar$ such that
\begin{itemize}
\item[(a)] (tuning) $G=G_0$ on $\cbar\sm V(\beta_0)$,
\item[(b)] (disk-annulus surgery) $G:\, A(\beta_0,\beta_1)\to U(\alpha)$ is a branched covering whose critical values are contained in $\phi^{-1}\circ\chi^{-1}(\PPP_Q)$,
\item[(c)] (new fixed critical point) $G:\, V(\beta_1)\to\cbar\sm\overline{U(\alpha)}$ is a branched covering with exactly one critical point at $c\in V$ and $G(c)=c$.
\end{itemize}
Then $G$ is a PCF branched covering with $\PPP_G=\phi^{-1}\circ\chi^{-1}(\PPP_Q)\cup\PPP_{R}\cup\{c\}$. Note that $\alpha$ is an essential curve in $\cbar\sm\PPP_G$. $G^{-1}(\alpha)$ contains exactly one essential curve $\alpha_0$ isotopic to $\alpha$ rel $\PPP_G$. From $G^{-1}(\alpha)=R^{-1}(\alpha)\cup\beta_1$, we know that $\deg G=\deg R+\deg G|_{\beta_1}$.

\begin{figure}[htbp]
\centering
\includegraphics[width=10cm]{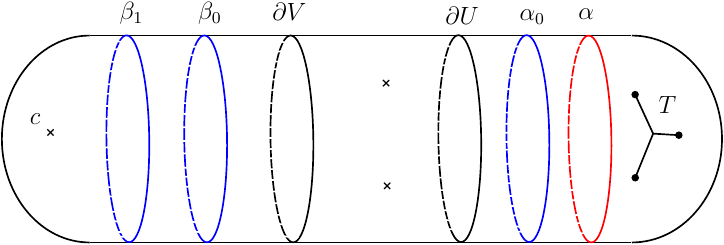}
\caption{Definition of $G$. }
\label{fig: G}
\end{figure}

\begin{lem}\label{lem:existence}
Assume that $G_0$ has no Thurston obstructions. Then $G$ is Thurston equivalent to a rational map $g$, which satisfies the conditions in Theorem \ref{thm:tunable}.
\end{lem}

\begin{proof}
By Lemma \ref{lem:irreducible}, we only need to prove that for any irreducible multicurve $\G$ of $G$, its leading eigenvalue $\lambda(M_{\G})<1$. Assume $k(\alpha,\g)=\#(\alpha\cap\g)$ for every $\g\in\G$. If there is a curve $\g\in\G$ such that $k(\alpha,\g)=0$, then $k(\alpha,\delta)=0$ for any curve $\delta$ in $G^{-1}(\g)$ since the curve $\alpha_0$ in $G^{-1}(\alpha)$ is isotopic to $\alpha$ rel $\PPP_G$. Note that $\G$ is irreducible. 
Thus $k(\alpha,\g)=0$ for each $\g\in\G$. Therefore either each curve in $\G$ is contained in $U(\alpha)$, or each curve in $\G$ is disjoint from $U(\alpha)$. In the former case, the transition matrix of $\G$ under $G$ is equal to the transition matrix of $\G$ under $G_0$. Hence $\lambda(M_{\G})<1$ by the assumption that $G_0$ has no Thurston obstructions. In the latter case, for each $\g\in\G$, $G^{-1}(\g)$ is disjoint from $A(\beta_0,\beta_1)$. Thus the transition matrix of $\G$ under $G$ is also equal to the transition matrix of $\G$ under $G_0$, which also implies $\lambda(M_{\G})<1$.

Now we assume $k(\alpha,\g)\neq 0$ for each $\g\in\G$. Using the same argument as in the proof of Theorem \ref{thm:tunable}, we know that if $\lambda(M_{\G})\ge 1$, then $\G$ contains a Levy cycle $\{\g_i\}$ since $Q$ is injective on its Hubbard tree. Moreover,
$$
\#((\beta_0\cup\beta_1)\cap\g'_i)+k(\alpha,\g_i)\le\#(G^{-1}(\alpha)\cap\g'_i)\le\#(\alpha\cap\g_{i+1})=k(\alpha,\g_{i+1}).
$$
Thus $\#((\beta_0\cup\beta_1)\cap\g'_i)=0$ and hence $G|_{\g'_i}=G_0|_{\g'_i}$. This is a contradiction since $G_0$ has no Thurston obstructions. In summary, we have $\lambda(M_{\G})<1$.

By Theorem \ref{thm:thurston}, $G$ is Thurston equivalent to a PCF rational map $g$ by a pair of orientation-preserving homeomorphisms $(\phi_0,\phi_1)$ of $\cbar$. Let $D$ be the fixed Fatou domain of $g$ containing the fixed critical point $\phi_0(c)$. Then $D$ is coiled by the polynomial $Q$. Thus the rational map $g$ satisfies the conditions in Theorem \ref{thm:tunable}.
\end{proof}
\begin{cor}\label{cor:example}
Applying Theorem \ref{thm:tunable} to the rational map $g$ in Lemma \ref{lem:existence}, the rational map $f$ in Theorem \ref{thm:tunable} has a completely stable multicurve containing coiling curves but no Cantor multicurves.
\end{cor}

\begin{proof}
We only need to show that $F$ in Theorem \ref{thm:tunable} has a multicurve containing coiling curves but no Cantor multicurves. 

Note that $\alpha_0$ is the unique curve in $g^{-1}(\alpha)$ isotopic to $\alpha$ rel $\PPP_g$ and $\beta_i$ ($i=0,1$) be all the curves in $g^{-1}(\alpha)$ which are peripheral around the fixed point $a\in D$ rel $\PPP_g$. After tuning, $\alpha_0$ is still essential in $\cbar\sm\PPP_F$, and $\beta_i$ $(i=0,1)$ are essential and isotopic to each other in $\cbar\sm\PPP_F$. All the other else curves in $g^{-1}(\alpha)=F^{-1}(\alpha)$ are still non-essential in $\cbar\sm\PPP_F$. Thus $\G:=\{\alpha,\beta_0\}$ is a completely stable multicurve of $F$. The set $F^{-1}(\beta_0)$ contains exactly one essential curve, which is isotopic to $\beta_0$ rel $\PPP_F$. The set $F^{-1}(\alpha_0)$ contains one curve isotopic to $\alpha_0$ rel $\PPP_F$ and two curves isotopic to $\beta_0$ rel $\PPP_F$. Thus $\kappa_n(\beta_0,\G)=2n+1\to\infty$ as $n\to\infty$ but $\kappa_n(\alpha_0,\G)=1$ for all $n\ge 1$. On the other hand, for $\G_0:=\{\alpha_0\}, \G_1:=\{\beta_0\}$, $\kappa_n(\alpha_0,\G_0)=\kappa_n(\beta_0,\G_1)=1$ for all $n\geq 1$. Hence $\G$ contains no Cantor multicurves. Thus $\G$ is a completely stable multicurve of $F$ containing coiling curves but no Cantor multicurves.
\end{proof}

\subsection{Examples of rational maps with the special coiled Fatou domains}\label{sec:specific}
To get a rational map which contains coiling curves but no Cantor multicurves, we need to find branched covering maps $G$ such that they have no Thurston obstructions. In this part, we provide some explicit examples of rational maps which are Thurston equivalent to $G$ and then automatically they have no Thurston obstructions. Then the resulting maps satisfy the conditions of Theorem \ref{thm:tunable}.

\vskip 0.24cm
\noindent
{\bf Example 1}. We consider
$$
R(z)=-\frac{27}{4}z^2(z-1)  \text{\quad and\quad} Q(z)=z^2-1.
$$
The cubic polynomial $R$ has two critical points $0$, $2/3$ and one critical orbit $2/3\mapsto 1\mapsto 0\mapsto 0$. Then $R$ has a unique fixed Fatou domain $U$ containing $0$ and $R^{-1}(U)=U\cup V$, where $V$ is the Fatou domain containing $1$. Since $R$ is a real polynomial and has a Fatou domain containing $2/3$ (which is different from $U$ and $V$), we have $\overline{U}\cap \overline{V}=\emptyset$ and hence all bounded Fatou domains of $R$ have pairwise disjoint closures. By \cite{shen2021primitive}, we obtain a cubic polynomial $g_0$ which is the tuning of $R$ with $Q$ in $U$ and
$$
g_0(z)=\frac{1+\nu}{\nu}z^2\left(z-\mu\right)+\nu \text{\quad with } \mu=\frac{1+\nu+\nu^2}{1+\nu},
$$
where $\nu=-0.1219358974084060...$ is chosen such that $g_0$ has two finite critical points $0$, $c_0=2\mu/3$ and one critical orbit $c_0\mapsto 1\mapsto 0\mapsto \nu\mapsto 0$. See Figure \ref{fig:exam1} for the Julia sets of $R$ and $g_0$.

\begin{figure}[htbp]
\setlength{\unitlength}{1mm}
\centering
\includegraphics[width=13cm]{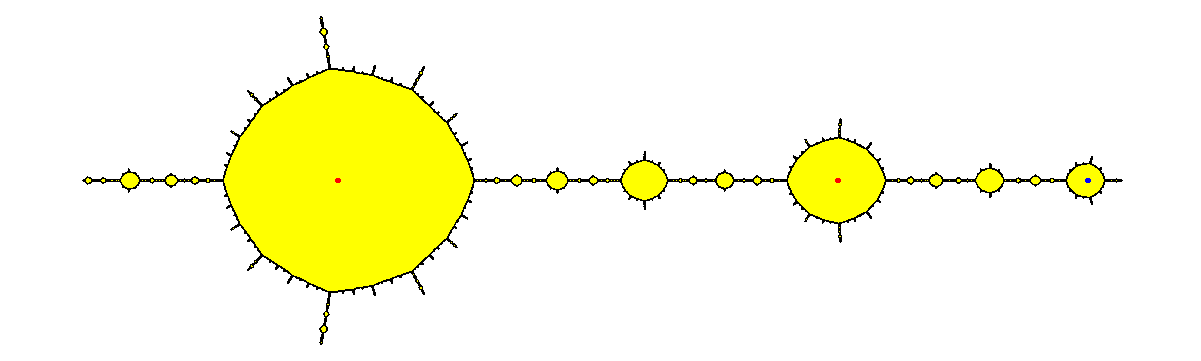}
\includegraphics[width=13cm]{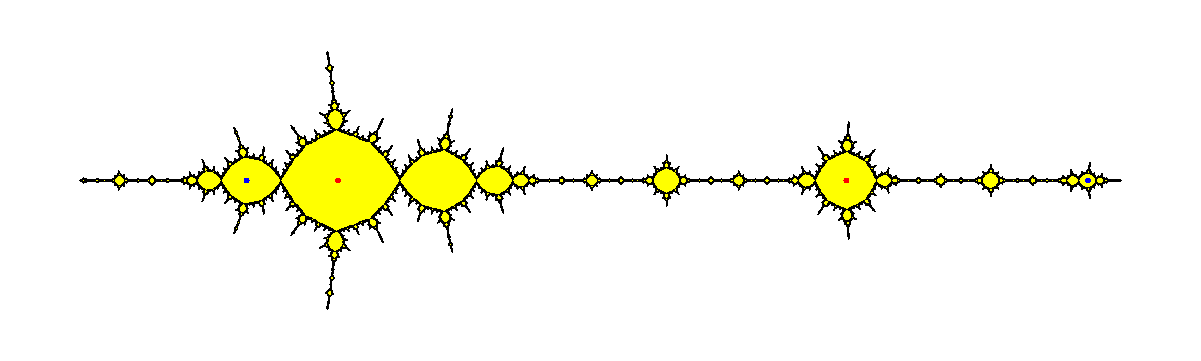}
\caption{Julia sets $\JJJ_R$, $\JJJ_{g_0}$ in Example 1. The critical orbits are marked. The ranges of these pictures are all $[-0.5,1.3]\times [-0.25,0.25]$.}
\label{fig:exam1}
\end{figure}

By performing a disk-annulus surgery and applying Lemma \ref{lem:existence}, we obtain a rational map $g$ with a coiling Fatou domain. The map $g$ has the form
$$
g(z)=-\frac{d(z-a)(z-b)(z-c)^3}{bc^3(z^2-q z+d)} \text{\quad with } q=\frac{d(3ab+ac+bc)}{abc},
$$
where
\begin{equation*}
\begin{split}
a&=- 0.1203582660251960...,  \quad b=0.1369645575161714..., \\
c&=\hskip0.33cm 0.9975907956140505...,  \quad d=1.0001239392656081...
\end{split}
\end{equation*}
are chosen such that $g$ has $8$ critical points: $\infty$ and $c$ are double, $0$, $1$, $c_1$ and $c_2$ are simple; and $4$ critical orbits:
$\infty\mapsto \infty$,
$c_2\mapsto 1\mapsto 1$,
$c\mapsto 0\mapsto a\mapsto 0$ and
$c_1\mapsto a\mapsto 0\mapsto a$.
See Figure \ref{fig:exam1-g} for the Julia set of $g$, where the Fatou domain of $g$ containing $1$ (colored green) is a coiled Fatou domain of $g$.

\begin{figure}[htbp]
\includegraphics[width=13cm]{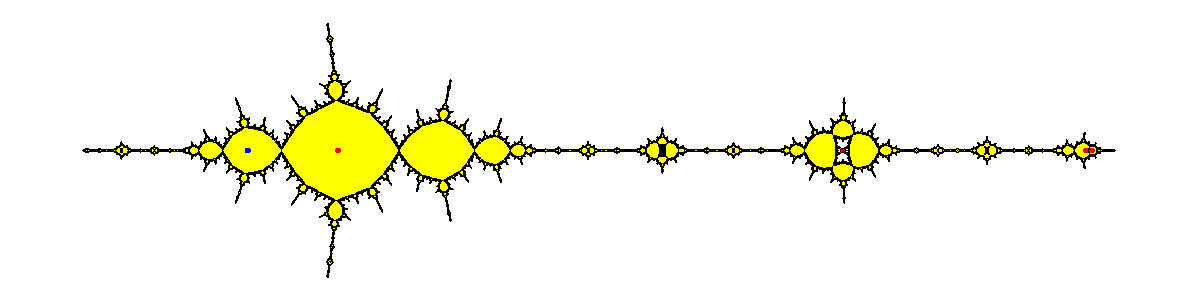}
\setlength{\unitlength}{1mm}
\setlength{\fboxsep}{0pt}
\centering
\fbox{\includegraphics[width=4.5cm]{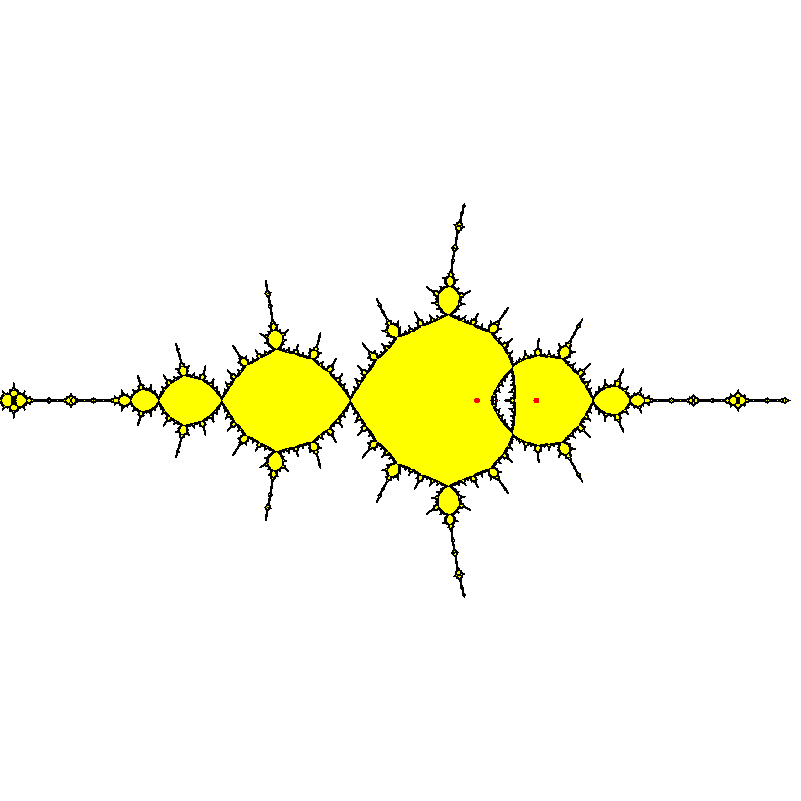}}\quad
\fbox{\includegraphics[width=4.5cm]{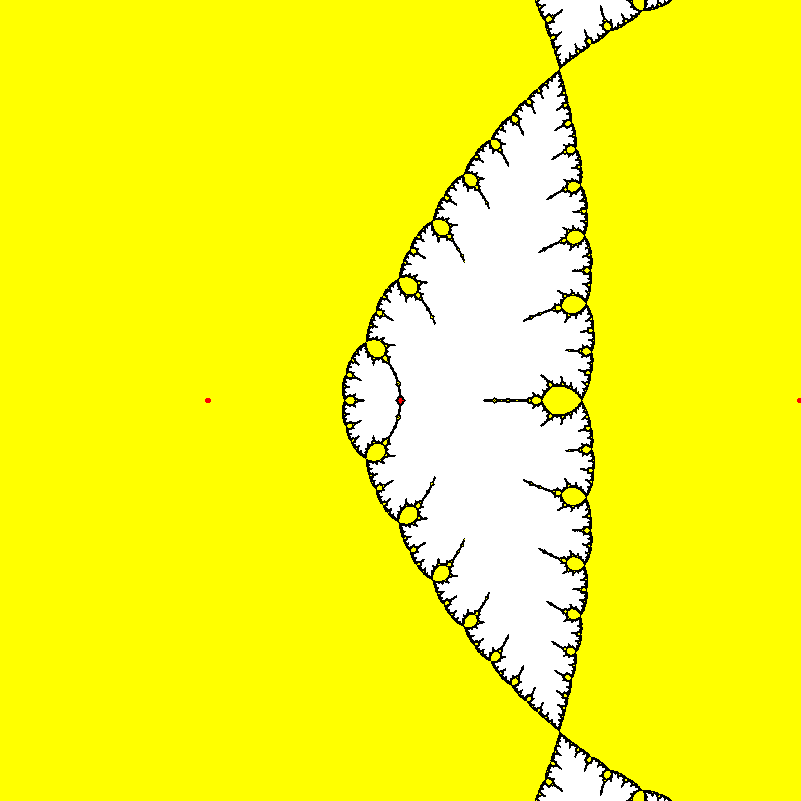}}\quad
\fbox{\includegraphics[width=4.5cm]{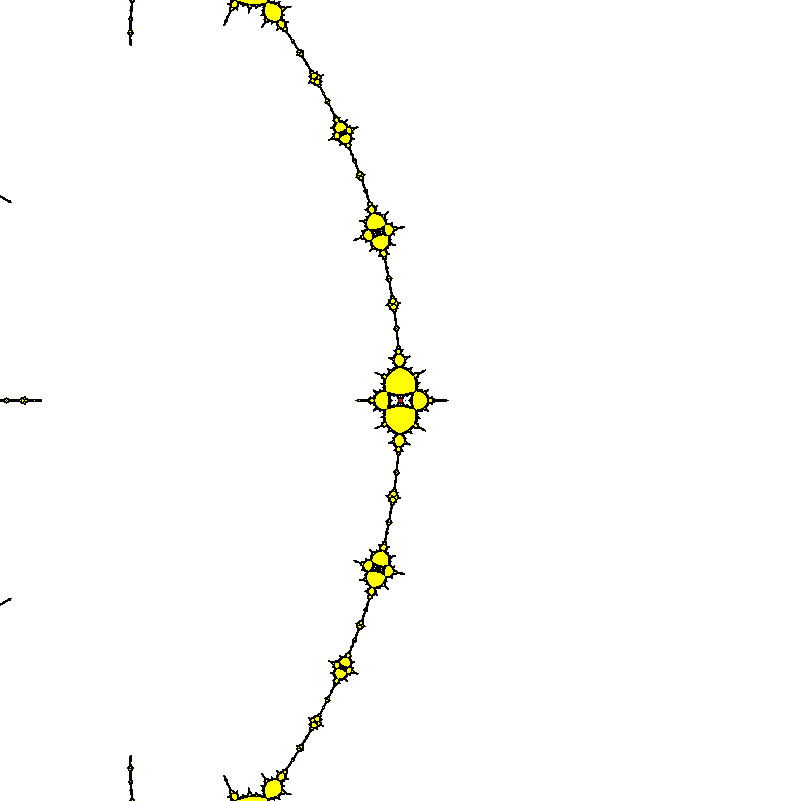}} \\ \medskip
\fbox{\includegraphics[width=4.5cm]{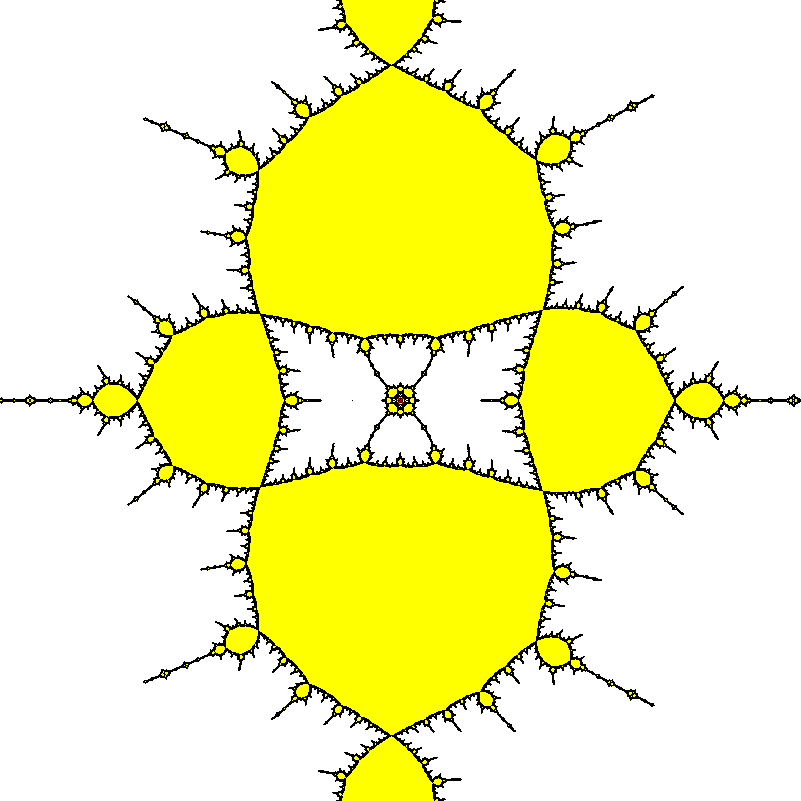}}\quad
\fbox{\includegraphics[width=4.5cm]{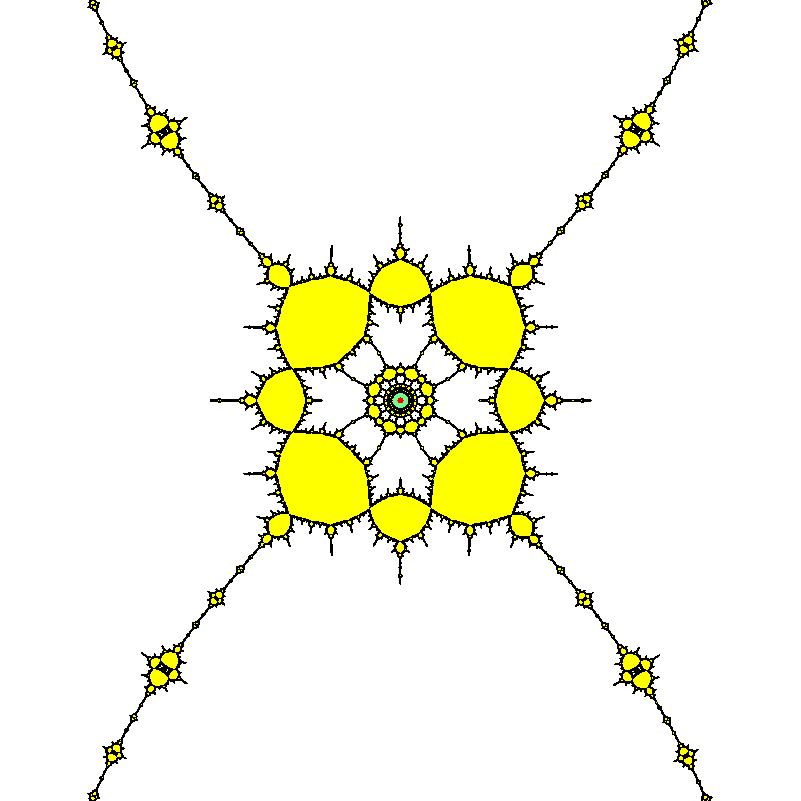}}\quad
\fbox{\includegraphics[width=4.5cm]{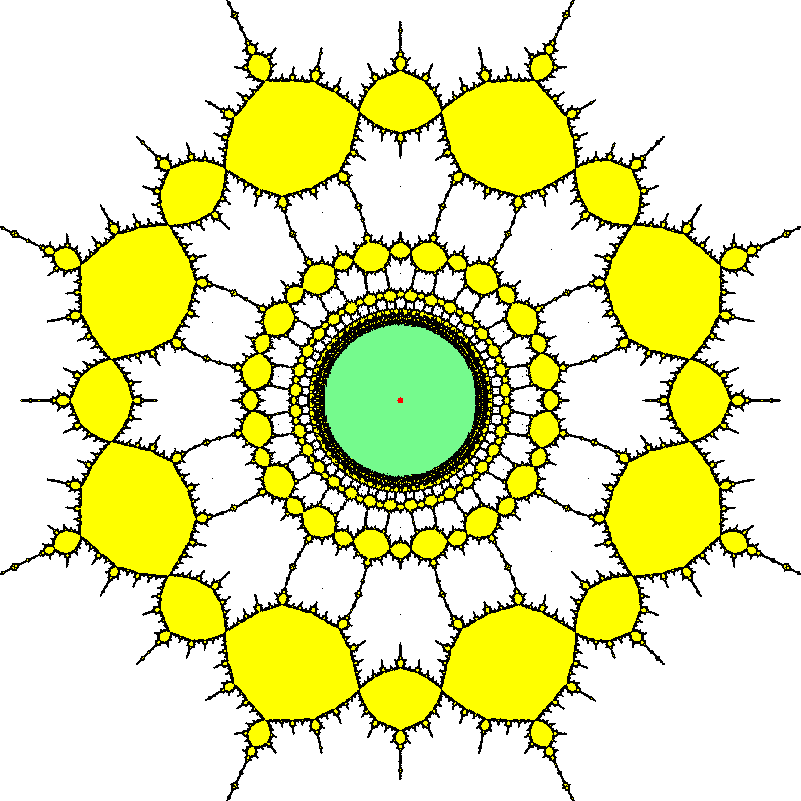}}
\caption{The Julia set $\JJJ_g$ and its successive zoom in near the fixed point $1$ of $g$ in Example 1. The widths of these pictures are $10^{-k}$, where $k=1,2,\cdots, 6$ respectively.}
\label{fig:exam1-g}
\end{figure}

\vskip 0.24cm
\noindent
{\bf Example 2}.
We consider
$$
R(z)=\frac{2z^2}{z^4+1}\text{\quad and\quad}
Q(z)=z^2-1.
$$
The rational map $R$ has $6$ simple critical points $\infty$, $0$, $\pm 1$, $\pm i$ and two critical orbits $\infty\mapsto 0$ and $\pm i\mapsto -1\mapsto 1\mapsto 1$.
Then $R$ is a PCF rational map with a fixed Fatou domain $U$ containing $0$ and $R^{-1}(U)=U\cup V$, where $V$ is the Fatou domain containing $\infty$.
After tuning $R$ with $Q$ (such tuning exists by \cite{cui2018hyperbolic}) we obtain a quartic rational map
$$
g_0(z)=\frac{z^2-\nu^2}{\nu(\nu-1)z^4-(2\nu^2-2\nu-1)z^2-\nu},
$$
where $\nu=-0.4287815744562657...$ is chosen such that the even map $g_0$ has $6$ critical points $\infty$, $0$, $\pm 1$, $\pm c_0=\pm i\sqrt{1-2\nu^2}$ and two critical orbits $\infty\mapsto 0\mapsto \nu\mapsto 0$ and $\pm c_0\mapsto -1\mapsto 1\mapsto 1$. See Figure \ref{fig:exam2} for the Julia sets of $R$ and $g_0$.

\begin{figure}[htbp]
\setlength{\unitlength}{1mm}
\centering
\includegraphics[width=6.5cm]{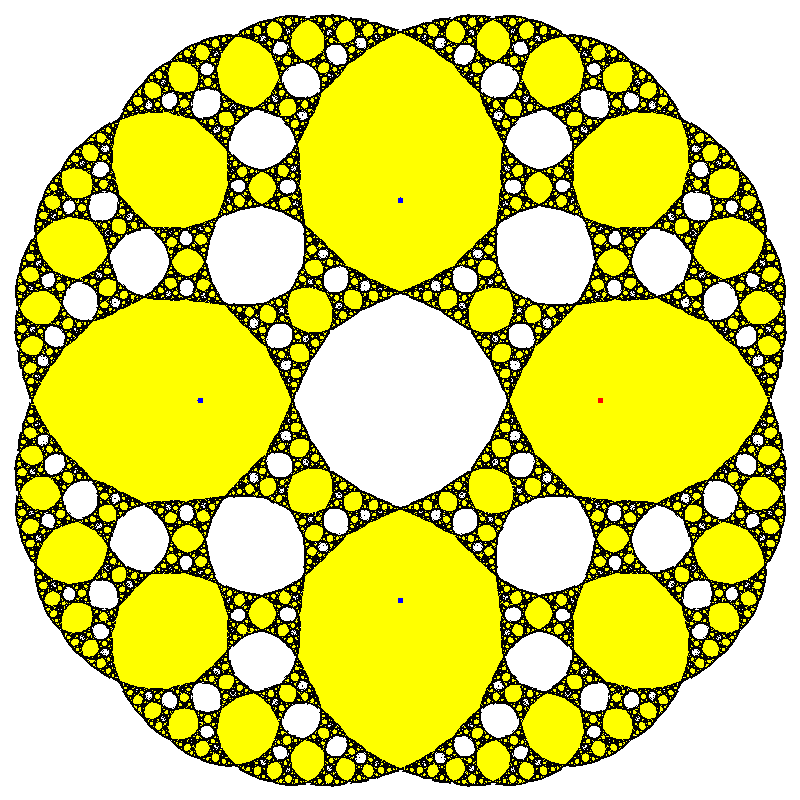} \qquad
\includegraphics[width=6.5cm]{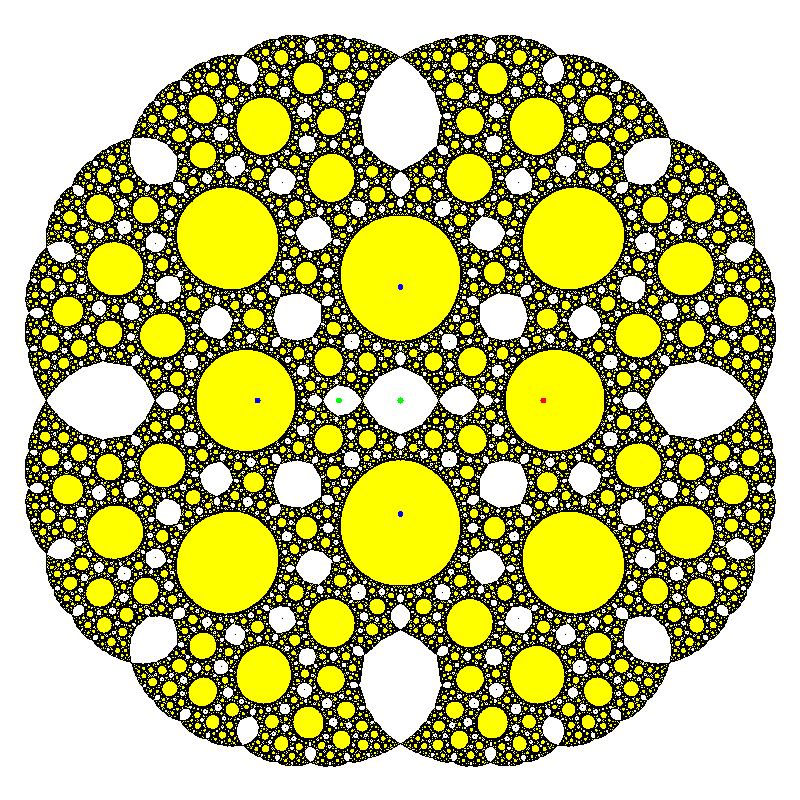}
\caption{Julia sets $\JJJ_R$ and $\JJJ_{g_0}$ in Example 2. The critical orbits are marked.}
\label{fig:exam2}
\end{figure}

\begin{figure}[htbp]
\centering
\includegraphics[width=6.5cm]{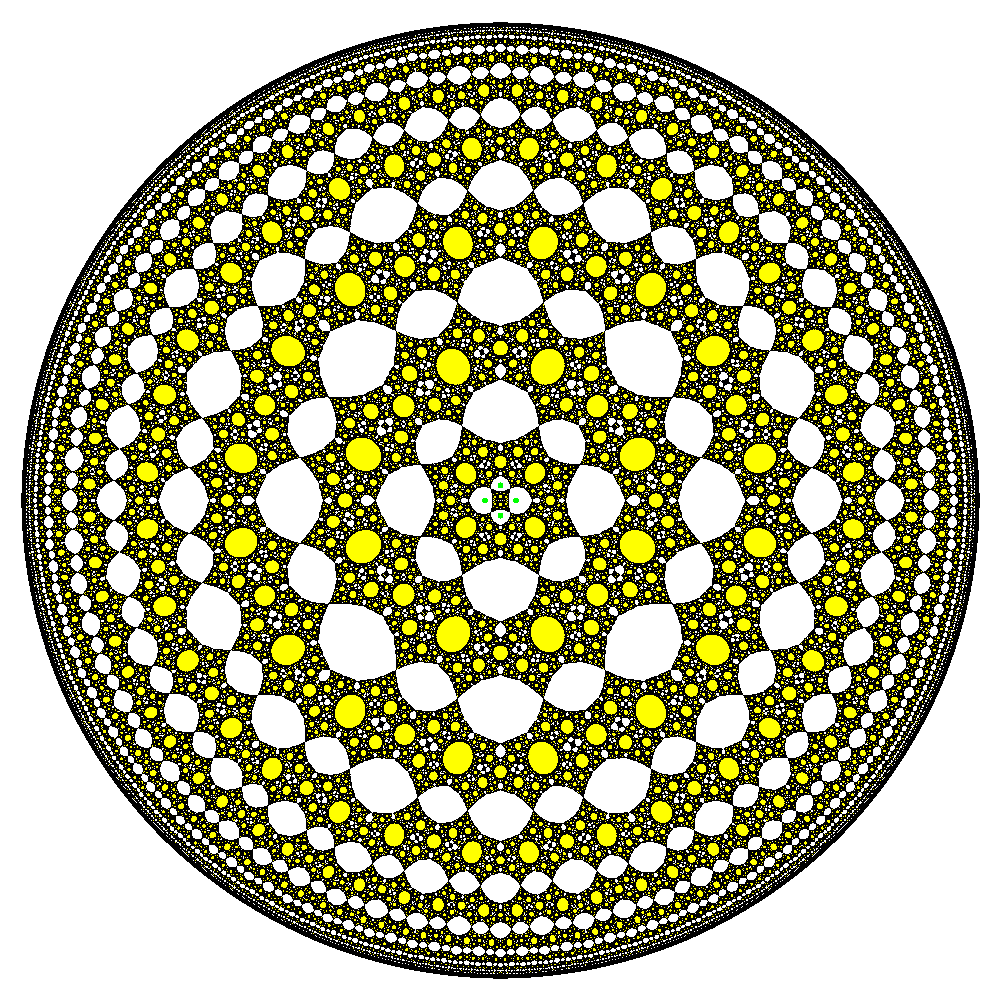} \qquad
\includegraphics[width=6.5cm]{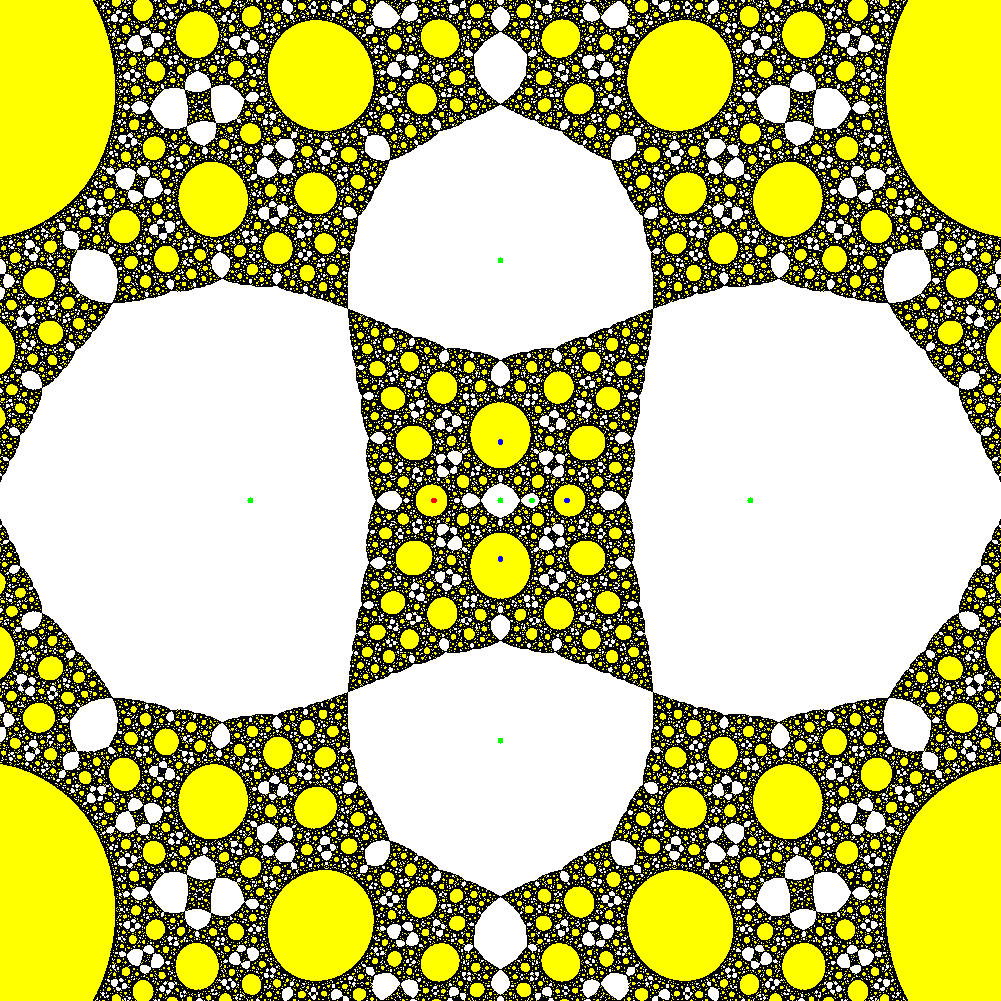}
\caption{The Julia set $\JJJ_g$ in Example 2 and its zoom in near the origin. The widths of these two pictures are $64$ and $4$ respectively.}
\label{fig:g0}
\end{figure}

By performing a disk-annulus surgery and applying Lemma \ref{lem:existence}, we obtain a rational map $g$ with a coiling Fatou domain. It has the form
$$
g(z)=\frac{c(z^2-1)^2(z^2-a^2)}{z^4+bz^2-ac},
$$
where
\begin{equation*}
\begin{split}
a&=\hskip0.33cm 0.1266022073620638...,  \quad b=-0.0469758128977771..., \\
c&=-0.0327926126839635...
\end{split}
\end{equation*}
are chosen such that $g$ has $10$ simple critical points $\infty$, $0$, $\pm 1$, $\pm c$, $\pm c_1$, $\pm c_2$ and $4$ critical orbits
$\infty\mapsto \infty$,
$\pm c\mapsto -c_1\mapsto c_1\mapsto c_1$,
$\pm 1\mapsto 0\mapsto a\mapsto 0$ and
$\pm c_2\mapsto a\mapsto 0\mapsto a$.
See Figure \ref{fig:g0} for the Julia set of $g$, where the unbounded domain is a coiled Fatou domain of $g$.

\bibliography{renorm}{}
\bibliographystyle{abbrv}
\end{document}